\documentclass[12pt, reqno]{amsart}
\usepackage{amssymb,amsthm,amsmath}
\usepackage[numbers,sort&compress]{natbib}
\usepackage{amsfonts}
\usepackage{mathrsfs}
\usepackage{latexsym}
\usepackage{tikz}
\usepackage{tkz-euclide}
\usepackage{float}
\usepackage{pdfsync}
\usepackage{indentfirst}
\date{June 18, 2021}

\usepackage[colorlinks=true, linkcolor=blue, citecolor=red]{hyperref}%在这里，我们设置了 colorlinks=true，这会让超链接的文本本身变成颜色，而不是一个方框。linkcolor 控制文内引用的颜色，citecolor 控制参考文献引用的颜色, 在colorlinks 前面加上backref,来检查参考文献的引用情况
\newcommand{\beq}{\begin{equation}}
	\newcommand{\eeq}{\end{equation}}
\newcommand{\ben}{\begin{eqnarray}}
	\newcommand{\een}{\end{eqnarray}}
\newcommand{\beno}{\begin{eqnarray*}}
	\newcommand{\eeno}{\end{eqnarray*}}
\numberwithin{equation}{section} % 设置公式编号与章节编号关联
\newtheorem{Thm}{Theorem}[section] % 定义新的定理环境，名为Thm，编号与章节编号关联
\newtheorem{Lem}[Thm]{Lemma} % 定义新的定理环境，名为Lem，与Thm共享编号
\newtheorem{Cor}[Thm]{Corollary}
\newtheorem{Rem}[Thm]{Remark}
\newtheorem{Def}[Thm]{Definition}

\numberwithin{equation}{section}

\allowdisplaybreaks

\begin{document}
	\title[Liouville type theorems for the MHD and Hall-MHD systems]{\bf Liouville type theorems for the 3D stationary MHD and Hall-MHD equations with non-zero constant vectors at infinity}
	
	\author{Wendong~Wang}
	\address[Wendong~Wang]{School of Mathematical Sciences, Dalian University of Technology, Dalian, 116024,  China}
	\email{wendong@dlut.edu.cn}
	
	\author{Guoxu~Yang}
	\address[Guoxu~Yang]{School of Mathematical Sciences, Dalian University of Technology, Dalian, 116024,  China}
	\email{guoxu\_dlut@outlook.com}

	\date{\today}
	\maketitle

	\begin{abstract}In this paper, we investigate Liouville type theorems for the three-dimensional steady-state MHD or Hall-MHD system under some asymptotic assumptions at infinity. Firstly, for the Hall-MHD system we obtain that $u$ and $B$ are constant vectors for any fluid viscosity, magnetic resistivity or Hall-coefficient when the magnetic field $B$ tends to a non-zero constant vector at infinity while the velocity field $u$ tends to $0$. Secondly, it also follows that $u$ and $B$ are constant for the Hall-MHD system when the velocity field tends to a constant vector at infinity while the magnetic field tends to $0$ without  any assumptions on viscosity,  magnetic resistivity or Hall-coefficient. One main difficulty lies in the Hall term, and we obtain the $L^p$ estimates of a generalized Oseen system with some supercritical terms via Lizorkin's theory and  prove that the operator is stable by exploring Kato's stability theorem. Moreover,  some similar results for the degenerate fluid viscosity or magnetic resistivity for the MHD system  are also obtained, which is independent of interest.
		%The main idea is based on the estimates of generalized Oseen system and Kato's stability theorem of bounded invertibility.
		%Our results also extend to the Liouville theorem of the MHD equation($\kappa$ and $\nu$ are not equal).
	\end{abstract}
	

	%\tableofcontents
	
	\setcounter{equation}{0}

	\section{Introduction}
	
	Consider the following general stationary MHD \& Hall-MHD system in $ \mathbb{R}^3$ : 
	\begin{equation}
		\label{eq:MHD}
		\left\{\begin{array}{l}
			-\kappa \Delta u+u \cdot \nabla u+\nabla p=B \cdot \nabla B, \\
			-\nu \Delta B+u \cdot \nabla B-B \cdot \nabla u=\alpha \nabla \times((\nabla \times B) \times B), \\
			\operatorname{div} u=0, \quad \operatorname{div} B=0,
		\end{array}\right.
	\end{equation}
	where $u =(u_1,u_2,u_3)$ is the velocity field of the fluid flows, $B  = (B_1,B_2,B_3)$ is the magnetic field, and $p $ is the pressure of the flows. In addition, $\kappa>0$, $\nu>0$, and $\alpha \geq 0$ are given parameters denoting the fluid viscosity, the magnetic resistivity, and the Hall-coefficient, respectively. When $\alpha = 0$,  the system (\ref{eq:MHD}) is reduced to the MHD system, which describes the steady state of the magnetic properties of electrically conducting fluids, including
	plasmas, liquid metals, etc;  for the physical background
	we refer to Schnack \cite{Sch} and the references therein. When $\alpha >0$,  this system governs the dynamics plasma flows of strong shear magnetic fields as in the solar flares, and has many important applications in the astrophysics (for example, see Chae-Degond-Liu \cite{CDL}). For more physical backgrounds of the Hall-MHD equations, we refer the readers to \cite{For,SU,BT} and the references therein.

	%	where $u$ is the velocity field of the fluid flows, $B$ is the magnetic field and $p$ is the pressure of the flows. If $\alpha > 0$, this Hall-MHD system plays an important role in many physical problems, such as magnetic reconnection in space plasmas, star formation, and neutron stars. Hall-MHD equations are derived strictly from Euler-Maxwell equations or kinetic model in \cite{acheritogaray2011kinetic}.
	%		If $\alpha=0$, this MHD system is a fundamental system of partial differential equations in nature, which describes the large scale and slow dynamics of plasmas. We refer readers to \cite{landau2013electrodynamics} for more background of the MHD system.
	
	In this paper, we focus on the Liouville-type properties of the system (\ref{eq:MHD}), which is
	motivated by the development of Navier-Stokes equations. When $\alpha=0$ and $B \equiv 0$, (\ref{eq:MHD}) is reduced to the Navier-Stokes system
	%\begin{equation}
	%			\label{ns}
	%			\left\{\begin{array}{l}
		%				u \cdot \nabla u+\nabla p-\kappa \Delta u=0, \\
		%				\nabla \cdot u=0 ,
		%			\end{array}\right.
	%		\end{equation}
and a very
challenging open question is whether there exists a nontrivial solution when the Dirichlet integral $\int_{\mathbb{R}^3}|\nabla u|^2dx$ is finite, which
dates back to Leray's celebrated paper \cite{Leray} and is explicitly written in Galdi's book (\cite{Galdi}, Remark X. 9.4; see also Tsai's book \cite{Tsai-2018}).
The Liouville type problem without any other assumptions is widely open. Galdi proved the above Liouville type theorem by assuming $u\in L^{\frac92}(\mathbb{R}^3)$ in \cite{Galdi}.
Chae \cite{Chae2014} showed the condition $\triangle u\in L^{\frac65}(\mathbb{R}^3)$ is sufficient for the vanishing property of $u$ by exploring the maximum principle of the head pressure. Chae-Wolf \cite{ChaeWolf} gave an improvement of logarithmic form for Galdi's result
by assuming that $\int_{\mathbb{R}^3} |u|^{\frac92}\{\ln(2+\frac{1}{|u|})\}^{-1}dx<\infty$.
Seregin obtained the conditional criterion $u\in BMO^{-1}(\mathbb{R}^3)$ in \cite{Se}.
For more references on conditional Liouville properties or Liouville properties on special domains, we refer to \cite{KTW,Se2,SeWang,CPZZ,Chae-Wolf2019,Chae2021, LRZ,BGWX} and the references therein.
% Moreover, Kozono-Terasawa-Wakasugi proved $u\equiv0$ if the vorticity \cite{SeWang} by Seregin and the author, where the following energy description was stated:
Relatively speaking, the two-dimensional case is easier because the vorticity of the 2D Navier-Stokes equations satisfies a nice elliptic equation, to which the maximum principle applies. For example, Gilbarg-Weinberger \cite{GW1978} obtained Liouville type theorem provided the Dirichlet energy is finite. As a different type of Liouville property for the 2D Navier-Stokes equations, Koch-Nadirashvili-Seregin-Sverak \cite{KNSS} showed that any bounded solution is a trivial solution as a byproduct of their results on the non-stationary case (see also \cite{BFZ2013} for the unbounded velocity).

However, for the MHD system, the situation is quite different. Due to the lack of maximum principle, there is not much progress in the study of MHD equation. For the 2D MHD equations, Liouville type theorems were proved by assuming the smallness of the norm of the magnetic field in \cite{WW}, and De Nitti et al. \cite{De} removed the smallness assumption. For the 3D MHD equations, Chae-Weng \cite{CW2016} proved that if the smooth solution satisfies D-condition ($\int_{\mathbb{R}^3}\left(|\nabla u|^2+|\nabla B|^2\right) d x<\infty$) and $u \in L^3(\mathbb{R}^3)$, then the solutions $(u, B)$ is identically zero. In \cite{Schulz}, Schulz proved that if the smooth solution $(u, B)$ of the stationary MHD equations is in $L^6(\mathbb{R}^3)$ and $u,B \in BMO^{-1}(\mathbb{R}^3)$ then it is identically zero. Chae-Wolf \cite{Chae-Wolf-MHD} showed that $L^6$ mean oscillations of the potential function of the velocity and magnetic field have certain linear growth by using the technique of \cite{Chae-Wolf2019}. Recently, Li-Pan \cite{LP} proved two forms of Liouville theorems for D-solutions, one is the case where $u\rightarrow(1,0,0)$ and $B\rightarrow0$ for any viscosity and magnetic resistivity; another case is $u\rightarrow0$ and $B\rightarrow(-1,0,0)$ by taking the same viscosity and magnetic resistivity. For more references, we refer to \cite{LLN,FW} and the references therein.

As said in \cite{CDL}: ``{\it Hall-MHD is believed to be an essential feature in the problem of magnetic reconnection. Magnetic 
	reconnection corresponds to changes in the topology of magnetic field lines which are ubiquitously 
	observed in space.}" It's interesting to investigate the mathematical theory of Hall-MHD system. Mathematical derivations of Hall-MHD equations from either two-fluids or kinetic models can be found in \cite{ADFL} and the first existence result of global weak solutions is given in \cite{CDL}.  
%In  \cite{CDPS},
%a stability analysis of a Vlasov equation modeling the Hall effect in plasmas is carried over.
The Liouville type problem of the Hall-MHD equations (\ref{eq:MHD}) is started by Chae-Degond-Liu in \cite{CDL}, where they proved that if a smooth solution $(u, B) \in L^{\infty}(\mathbb{R}^3) \cap L^{\frac{9}{2}}\left(\mathbb{R}^3\right)$ and satisfying the D-condition, then $u=B=0$, which was improved in \cite{ZYQ} by removing the bounded-ness condition. Chae-Weng \cite{CW2016} proved that if $u \in L^3(\mathbb{R}^3)$ and with the D-condition for the equations (\ref{eq:MHD}), then $u = B = 0$. Chen-Li-Wang improved this result by only assuming $u \in BMO^{-1}(\mathbb{R}^3)$ and $B\in L^{6,\infty}(\mathbb{R}^3)$ in \cite{CLW} via  $L^q$ estimates of stationary Stokes system.  Recently,  Cho-Neustupa-Yang \cite{CNY} established some new Liouville-type theorems for weak solutions under some conditions on the growth of certain Lebesgue norms of the velocity and the magnetic field. For more references, we refer to \cite{Chae-Wolf-MHD, YX} and the references therein.

Motivated by the mathematical theory of the Oseen system in \cite{Galdi} and the progress of the MHD equation in \cite{LP},  we will focus on studying the Liouville properties of the Hall-MHD system. 
%In the following text, let $\alpha > 0$ unless otherwise stated. 

One of  our main results is stated as follows.
\begin{Thm}
	\label{the1.1}
	Let $(u, B)$ be a smooth solution to the Hall-MHD equations (\ref{eq:MHD}) satisfying the finite Dirichlet integral (or D-condition)
	\ben\label{eq:d-condition}
	\int_{\mathbb{R}^3}\left(|\nabla u|^2+|\nabla B|^2\right) d x<\infty,
	\een
	and the asymptotic behavior at infinity
	$$
	u(x) \rightarrow u_{\infty}, \quad B(x) \rightarrow B_{\infty}, \quad \text { uniformly as } \quad|x| \rightarrow \infty,
	$$
	where $u_{\infty} = 0$ and $B_{\infty} = (-1, 0, 0)$. Then it follows that $(u, B) \equiv\left(u_{\infty}, B_{\infty}\right)$ provided that $\alpha \| \nabla B\|_{L^{\infty}\left(\mathbb{R}^3\right)} < \infty$.
\end{Thm}
\begin{Rem}
	(i) Finn \cite{Fin} and Ladyzhenskaya \cite{Lady} showed that any D-solution of Navier-Stokes equations in a 3D exterior domain converges to a prescribed constant vector $u_{\infty}$ at infinity. It's more difficult in 2D and Korobkov-Pileckas-Russo proved this result in \cite{KPR}. Moreover, the derivatives of the velocity and the pressure also converge to the constant in 3D, for example seeing Theorem X.5.1 in \cite{Galdi}. Hence the condition of $\| \nabla B\|_{L^{\infty}\left(\mathbb{R}^3\right)} < \infty$ seems to be natural.\\
	(ii)The above theorem shows that the Liouville type theorems hold for (\ref{eq:MHD}) without any assumptions on viscosity,  magnetic resistivity or the Hall-coefficient. One main difficulty lies in the Hall term
	$$
	\nabla \times((\nabla \times B) \times B)=B\cdot \nabla (\nabla \times B)-(\nabla \times B)\cdot \nabla B,
	$$
	which shares the same order as the main term of the Laplacian operator. One of our observations is combining the  Laplace operator and $B\cdot \nabla (\nabla \times B)$ as the main terms, and the other term of $(\nabla \times B)\cdot \nabla B$ as the perturbation term.
\end{Rem}

When the Hall-coefficient  $\alpha=0$, we immediately obtain the following corollary.
\begin{Cor}
	Let $(u, B)$ be a smooth solution to the MHD system with $\alpha = 0$ in (\ref{eq:MHD}) satisfying  the D-condition \eqref{eq:d-condition} and
	$$
	u(x) \rightarrow u_{\infty}, \quad B(x) \rightarrow B_{\infty}, \quad \text { uniformly as } \quad|x| \rightarrow \infty,
	$$
	where $u_{\infty}, B_{\infty} \in \mathbb{R}^3$ are two constant vectors. Then $(u, B) \equiv\left(u_{\infty}, B_{\infty}\right)$ follows if $u_{\infty} = 0, B_{\infty} \neq 0 $.
\end{Cor}

\begin{Rem}
	The above result improved (ii) of Theorem 1.1  in \cite{LP}, where the velocity viscosity is equal to the magnetic resistivity, i.e.  $\kappa=\nu$. One can assume $B_{\infty}=(-1,0,0)$, since the MHD system  is invariant under the orthogonal coordinate transform.
\end{Rem}

Moreover, when $u(x) \rightarrow u_{\infty}\neq 0$, we have the following conclusion.
\begin{Thm}
	\label{the1.4}
	Let $(u, B)$ be a smooth solution to the Hall-MHD equations (\ref{eq:MHD}) satisfying the D-condition \eqref{eq:d-condition} 
	and
	$$
	u(x) \rightarrow u_{\infty}, \quad B(x) \rightarrow B_{\infty}, \quad \text { uniformly as } \quad|x| \rightarrow \infty,
	$$
	where $u_{\infty} = (1, 0, 0)$ and $B_{\infty} = 0$. Then $(u, B) \equiv\left(u_{\infty}, B_{\infty}\right)$ follows if  $\alpha \| \nabla B\|_{L^{\infty}\left(\mathbb{R}^3\right)} < \infty$.
\end{Thm}
%	\begin{Rem}
	%		In Theorem \ref{the1.2}, we can take the case of $\kappa \neq \nu$ but $|\kappa-\nu|$ small(be treated as a perturbation). In addition, $|\kappa-\nu|$ small is reasonable because in high temperature, both the viscosity coefficient $\kappa$ and resistivity coefficient $\nu$ are usually very small \cite{califano1999resistivity}.
	%	\end{Rem}
\begin{Rem} Especially, when $\alpha=0$, the above theorem implies the result (i) of Theorem 1.1 in \cite{LP}. To deal with Hall-term, it seems that $\| \nabla B\|_{L^{\infty}\left(\mathbb{R}^3\right)} < \infty$ is necessary. In fact, one can replace $\| \nabla B\|_{L^{\infty}\left(\mathbb{R}^3\right)} < \infty$ by $\| \nabla  B\|_{L^{\frac{12}{5}}\left(\mathbb{R}^3\right)} < \infty$ by careful computation. Then, we have $\|(\nabla \times B)\cdot\nabla B\|_{L^{\frac{6}{5}}(\mathbb{R}^3)} <\infty$. See (\ref{equ3.5}) and (\ref{equ3.9}) for more details. It's interesting whether the condition of $ \| \nabla B\|_{L^{\infty}\left(\mathbb{R}^3\right)} < \infty$ or $\| \nabla  B\|_{L^{{\frac{12}{5}}}\left(\mathbb{R}^3\right)} < \infty$ can be removed if $\alpha\neq0$.
\end{Rem}
Furthermore, we consider the Liouville type theorems for the following stationary MHD system in $\mathbb{R}^3$ with degenerate viscosity or resistivity in different directions: 
\begin{equation}
	\label{equ1.3}
	\left\{\begin{array}{l}
		-\Delta_\kappa  u+u \cdot \nabla u+\nabla p=B \cdot \nabla B, \\
		-\Delta_\nu B+u \cdot \nabla B-B \cdot \nabla u= 0,\qquad  \\
		\operatorname{div} u=0, \quad \operatorname{div} B=0,
	\end{array}\right.
\end{equation}
where \ben\label{eq:laplace k}\Delta_\kappa:=\kappa_1\partial_{x_1}^2 + \kappa_2\partial_{x_2}^2 + \kappa_3\partial_{x_3}^2,\quad \Delta_\nu:=\nu_1\partial_{x_1}^2 + \nu_2\partial_{x_2}^2 + \nu_3\partial_{x_3}^2\een
with  $\kappa_1, \nu_1\geq 0$
and $\kappa_2,\kappa_3,\nu_2,\nu_3>0$.
\begin{Thm}
	\label{the1.6}
	Let $(u, B)$ be a smooth solution to the MHD equations (\ref{equ1.3})-\eqref{eq:laplace k} satisfying  the D-condition \eqref{eq:d-condition} 
	and
	$$
	u(x) \rightarrow u_{\infty}, \quad B(x) \rightarrow B_{\infty}, \quad \text { uniformly as } \quad|x| \rightarrow \infty,
	$$
	where $u_{\infty}, B_{\infty} \in \mathbb{R}^3$ are two constant vectors. Then $(u, B) \equiv\left(u_{\infty}, B_{\infty}\right)$ follows if one of the following conditions holds:\\
	(i) $u_{\infty}=(1,0,0), B_{\infty} = 0 $ with $\kappa_1, \nu_1\geq 0$;
	\\
	(ii) $u_{\infty}=0, B_{\infty} =  (-1,0,0)$ with $\kappa_1\nu_1>0$ or $\kappa_1=\nu_1=0.$
\end{Thm}

\begin{Rem}Note that when $\kappa_1 + \kappa_2 + \kappa_3 > 0$ and $ \nu_1 + \nu_2 + \nu_3 > 0$, it follows that $u=B\equiv0$ if 
	$B \in L^6(\mathbb{R}^3)$ and $u \in L^3(\mathbb{R}^3)$ (see Theorem 1.5, \cite{CLW}). Here it seems difficult to prove the Liouville theorem by only assuming that one of the second-order derivatives does not vanish. When all $\kappa_i=\nu_j=0$ for $i,j=1,2,3$, it is similar to the three-dimensional Euler equations. It is impossible to deduce that $u$ is vanishing from the  bounded-ness of $L^q$ norm. For example, we can refer to the counterexample belonging to $C_0^\infty(\mathbb{R}^3)$ in \cite{Ga}. 
\end{Rem}

This paper is organized as follows: some notations and some necessary lemmas are presented in Sect.\ref{sec2}; The Liouville type theorem on the Hall-MHD system with $B_\infty\neq0$ is obtained in Sect.\ref{sec3}, where we give the detailed proof of Theorem \ref{the1.1}; Theorem \ref{the1.4} is proved in  Sect.\ref{sec4}  on the Hall-MHD system with $u_\infty\neq0$; Sect.\ref{sec5} is devoted to the proof of Theorem \ref{the1.6} on the degenerate MHD system.

Throughout this paper, $C\left(c_1, c_2, \ldots, c_n\right)$ denotes a positive constant depending on $c_1, c_2, \ldots c_n$ which may be different from line to line. We denote $A' \leq C B'$ by $A' \lesssim B'$.

\section{Preliminaries}
\label{sec2}
First we introduce some notations. Denote by $B_r\left(x_0\right):=\left\{x \in \mathbb{R}^3:\left|x-x_0\right|<r\right\}$ and $B_r:=B_r(0)$. Denote by $\nabla^\gamma:=\partial_{x_1}^{\gamma_1} \partial_{x_2}^{\gamma_2} \partial_{x_3}^{\gamma_3}$, where  $\gamma=\left(\gamma_1, \gamma_2, \gamma_3\right)$, $\gamma_1, \gamma_2, \gamma_3 \in \mathbb{N} \cup\{0\}$, $\partial_i = \frac{\partial}{\partial x_i}$ and $|\gamma|=\gamma_1+\gamma_2+\gamma_3$. We denote $L^p(\Omega)$ by the usual Lebesgue space with the norm
$$
\|f\|_{L^p(\Omega)}:= \begin{cases}\left(\int_{\Omega}|f(x)|^p d x\right)^{1 / p}, & 1 \leq p<\infty, \\ \underset{x \in \Omega}{\operatorname{ess\, sup}}|f(x)|, & p=\infty,\end{cases}
$$
where $\Omega \subset \mathbb{R}^3, 1 \leq p \leq \infty$. $W^{k, p}(\Omega)$  and $\dot{W}^{k, p}(\Omega)$ are defined as follows:
$$
\begin{aligned}
	\|f\|_{W^{k, p}(\Omega)} & :=\sum_{0 \leq|\gamma| \leq k}\left\|\nabla^\gamma f\right\|_{L^p(\Omega)}, \\
	\|f\|_{\dot{W}^{k, p}(\Omega)} & :=\sum_{|\gamma|=k}\left\|\nabla^\gamma f\right\|_{L^p(\Omega)},
\end{aligned}
$$
respectively. $C^{\infty}(\Omega)$ denotes the space of smooth functions on $\Omega$. $\mathcal{S}\left(\mathbb{R}^n\right)$ denotes the space of rapid decreasing smooth functions on $\mathbb{R}^n$. $\mathcal{P}_n$ stands for the space of polynomials in $\mathbb{R}^3$ with their degree no bigger than $n$.

We introduce the following multiplier theorem by Lizorkin (see \cite{Liz1963}, \cite{Liz1967}, or Section VII.4 in \cite{Galdi}), which plays an important role in the proof of the main theorems.
\begin{Thm}[Lizorkin]\label{thm:lizorkin}
	Let $\Phi: \mathbb{R}^n \rightarrow \mathbb{C}$ be continuous together with the derivative
	$$
	\frac{\partial^n \Phi}{\partial \xi_1 \ldots \partial \xi_n}
	$$
	and all preceding derivatives for $\left|\xi_i\right|>0, i=1, \ldots, n$. Then, if for some $\beta \in[0,1)$ and $M>0$
	\begin{equation}
		\label{eq:bound of phi}
		\left|\xi_1\right|^{\kappa_1+\beta} \cdot \ldots \cdot\left|\xi_n\right|^{\kappa_n+\beta}\left|\frac{\partial^\kappa \Phi}{\partial \xi_1^{\kappa_1} \ldots \partial \xi_n^{\kappa_n}}\right| \leq M,
	\end{equation}
	where $\kappa_i$ is zero or one and $\kappa=\sum_{i=1}^n \kappa_i=0,1, \ldots n$, the integral transform
	$$
	T u =\frac{1}{(2 \pi)^{n / 2}} \int_{\mathbb{R}^n} e^{i \boldsymbol{x} \cdot \boldsymbol{\xi}} \Phi(\xi) \widehat{u}(\xi) d \xi, \quad u \in \mathcal{S}\left(\mathbb{R}^n\right)
	$$
	defines a bounded linear operator from $L^q\left(\mathbb{R}^n\right)$ into $L^r\left(\mathbb{R}^n\right), 1<q<$ $\infty, 1 / r=1 / q-\beta$, and we have
	$$
	\|T u\|_{L^r\left(\mathbb{R}^n\right)} \leq C (q, \beta, M)\|u\|_{L^q\left(\mathbb{R}^n\right)}.
	$$
\end{Thm}

In order to deal with the perturbation system of the generalized Oseen system, we introduce the definition of the relative bounded-ness and Kato's stability theorem of bounded invertibility \cite{Kato}. 
\begin{Def}
	Let $\mathcal{T}$ and $\mathcal{A}$ be operators with the same domain space $\mathcal{X}$ (but not necessarily with the same range space) such that $\mathrm{D}(\mathcal{T}) \subset \mathrm{D}(\mathcal{A})$ and
	\begin{equation}
		\label{def2.13}
		\|\mathcal{A} u\| \leq a\|u\|+b\|\mathcal{T} u\|, \quad u \in \mathrm{D}(\mathcal{T}),
	\end{equation}
	where $a, b$ are nonnegative constants. Then we shall say that $\mathcal{A}$ is relatively bounded with respect to $\mathcal{T}$ or simply $\mathcal{T}$-bounded.
\end{Def}

\begin{Thm}[Kato]
	\label{lem2.5}
	Assume $\mathcal{X}$ and $\mathcal{Y}$ are Banach spaces. Let $\mathcal{T}$ and $\mathcal{A}$ be operators from $\mathcal{X}$ to $\mathcal{Y}$. Let $\mathcal{T}^{-1}$ exist and belong to $\mathscr{B}(\mathcal{Y}, \mathcal{X})$ ($\mathscr{B}(\mathcal{Y}, \mathcal{X})$ is the set of all bounded operators on $\mathcal{Y}$ to $\mathcal{X} $ and so that $\mathcal{T}$ is closed). Let $\mathcal{A}$ be $\mathcal{T}$-bounded, with the constants $a, b$ in (\ref{def2.13}) satisfying the inequality
	$$
	a\left\|\mathcal{T}^{-1}\right\|+b<1 .
	$$
	Then $\mathcal{S}=\mathcal{T}+\mathcal{A}$ is closed and invertible, and 
	\ben\label{eq:kato}
	\left\|\mathcal{S}^{-1}\right\|\leq \frac{\left\|\mathcal{T}^{-1}\right\|}{1-a\left\|\mathcal{T}^{-1}\right\|-b}.
	\een
\end{Thm}

%		\begin{Rem}
	%			We know that (\ref{eq:perturbation oseen}) is a perturbation of (\ref{eq:mixed oseen}) and (\ref{equ2.16}) is a perturbation of (\ref{equ2.12}). In fact, by augmenting the linear system in Lemma \ref{lem2.2} and Lemma \ref{lem2.3} with several terms, half of which are small and the other half controlled, we can obtain some results similar to Lemma \ref{lem2.6} and Lemma \ref{lem2.7}.
	%		\end{Rem}

\section{Hall-MHD system with $B_\infty\neq0$: Proof of Theorem \ref{the1.1}}

Consider 
the 3D stationary  generalized Oseen system with some special critical terms in   $\mathbb{R}^3$ as follows:
\begin{equation}
	\label{eq:mixed oseen}
	\left\{\begin{array}{l}
		\partial_{x_1} w_1+\nabla p-\mu_1 \Delta w_1 - \mu_2\Delta w_2 - \alpha^{\prime}\partial_{x_1}(\nabla\times(w_1-w_2)) =f_1, \\
		-\partial_{x_1} w_2+\nabla p-\mu_1\Delta w_2 - \mu_2\Delta w_1 - \alpha^{\prime}\partial_{x_1}(\nabla\times(w_2-w_1))=f_2,  \\
		\nabla \cdot w_1=g.
	\end{array}\right.
\end{equation}
Applying Theorem \ref{thm:lizorkin}, we derive a type of $L^q$ estimate for the above mixed Oseen system of (\ref{eq:mixed oseen}).
\begin{Lem}
	[$L^q$ estimate for generalized Oseen system]\label{lem:mixed ossen}
	Let $f_1, f_2 \in L^q\left(\mathbb{R}^3\right), g \in W^{1, q}\left(\mathbb{R}^3\right)$ with $1<q<2$ and $r=\left(\frac{1}{q}-\frac{1}{2}\right)^{-1}$. Moreover, $\mu_1, \mu_2$ and $\alpha^{\prime}$ are constants satisfying $\mu_1>0$ and $\mu_1>|\mu_2|$. Then 
	there exists a unique solution $\left(w_1, w_2, p\right)\in\left(\dot{W}^{2, q}\left(\mathbb{R}^3\right) \cap L^r\left(\mathbb{R}^3\right)\right)^6 \times\left(\dot{W}^{1, q}\left(\mathbb{R}^3\right) / \mathcal{P}_0\left(\mathbb{R}^3\right)\right)$ of the system (\ref{eq:mixed oseen}) such that
	\ben
	\label{eq:w1w2 estimate}
	&& \left\|\left(\nabla^2 w_1, \nabla^2 w_2\right)\right\|_{L^q\left(\mathbb{R}^3\right)}+\left\|\left(w_1, w_2\right)\right\|_{L^r\left(\mathbb{R}^3\right)}+\|\nabla p\|_{L^q\left(\mathbb{R}^3\right)} \nonumber\\
	& \leq &C(q, \mu_1, \mu_2)\left(\left\|f_1\right\|_{L^q\left(\mathbb{R}^3\right)}+\left\|f_2\right\|_{L^q\left(\mathbb{R}^3\right)}+\|g\|_{W^{1, q}\left(\mathbb{R}^3\right)}\right) .
	\een
\end{Lem}	
\begin{proof}
	By performing a Fourier transform on the system of \eqref{eq:mixed oseen}, we obtain that
	\begin{equation}
		\label{eq:w1w2}
		\left\{\begin{array}{l}
			(\mu_1 |\xi|^2 + i\xi_1) \hat{w_1} + \mu_2|\xi|^2\hat{w_2} + \alpha^{\prime}\xi_1\xi\times(\hat{w_1} - \hat{w_2}) + i\xi \hat{p} = \hat{f_1},\\
			(\mu_1 |\xi|^2 - i\xi_1) \hat{w_2} + \mu_2|\xi|^2\hat{w_1} -  \alpha^{\prime}\xi_1\xi\times(\hat{w_1} - \hat{w_2}) + i\xi \hat{p} = \hat{f_2},\\
			i \xi \cdot \hat{w_1}=\hat{g}.
		\end{array}\right.  
	\end{equation}
	Taking $i\xi$ on the both sides of  $(\ref{eq:w1w2})_1$ and $(\ref{eq:w1w2})_2$ as the vector inner product, we get
	\begin{equation}
		\label{equ2.5}
		\left\{\begin{array}{l}
			(\mu_1 |\xi|^2 + i\xi_1) \hat{g} + \mu_2|\xi|^2 i\xi\cdot\hat{w_2}-|\xi|^2\hat{p} = i\xi\cdot\hat{f_1},\\
			(\mu_1 |\xi|^2 - i\xi_1) i\xi\cdot \hat{w_2} + \mu_2|\xi|^2\hat{g} -|\xi|^2 \hat{p} = i\xi\cdot\hat{f_2},
		\end{array}\right. 
	\end{equation}
	which implies
	\begin{equation}
		\label{eq:pressure}
		\hat{p} = \frac{1}{(\mu_1 - \mu_2)|\xi|^2 - i\xi_1}\left\{\left[(\mu_1^2 - \mu_2^2)|\xi|^4 +\xi_1^2\right]\frac{\hat{g}}{|\xi|^2} - (\mu_1|\xi|^2 - i\xi_1)\frac{i\xi\cdot\hat{f_1}}{|\xi|^2} + \mu_2i\xi\cdot\hat{f_2}\right\}.
	\end{equation}
	
	Next we focus on solving the system of (\ref{eq:w1w2}). For simplicity, denote 
	\ben\label{eq:symbel}
	&&A_1:=\mu_1 |\xi|^2 + i\xi_1, \quad A_2:=\mu_1 |\xi|^2 - i\xi_1, \quad C_1:=\alpha^{\prime}\xi_1\xi;\nonumber\\
	&&F_1:=\hat{f_1}-i\xi\hat{p}, \quad F_2:=\hat{f_2}-i\xi\hat{p}, \quad E := \mu_2|\xi|^2.
	\een
	Taking the sum of  $(\ref{eq:w1w2})_1$ and $(\ref{eq:w1w2})_2$ , we get
	\begin{equation}
		\label{equ2.7}
		\hat{w_2} = -\frac{(A_1+E)}{A_2+E}\hat{w_1}+\frac{F_1+F_2}{A_2+E}.
	\end{equation}
	By substituting (\ref{equ2.7}) into $(\ref{eq:w1w2})_1$, we get
	\begin{equation}
		\label{eq:G}
		(A_1A_2-E^2)\hat{w_1} + (A_1+A_2+2E)C_1\times\hat{w_1} = A_2F_1 -EF_2+ C_1\times(F_1+F_2) :=G.
	\end{equation}
	Noting that $C_1:=\alpha^{\prime}\xi_1\xi$ and
	\begin{equation}
		\label{equ2.9}
		-|\xi|^2\hat{w_1} = i\xi\hat{g} + \xi\times(\xi\times\hat{w_1}),
	\end{equation}
	by taking cross product on the both sides of  (\ref{eq:G}) by $\xi$ it follows that 
	\begin{equation}
		\label{equ2.10}
		(A_1A_2-E^2)\xi\times\hat{w_1} - (A_1+A_2+2E)\alpha^{\prime}\xi_1|\xi|^2\hat{w_1}= \xi\times G + i\alpha^{\prime}\xi_1 (A_1+A_2+2E)\xi\hat{g}.
	\end{equation}
	Combining (\ref{eq:G}) and (\ref{equ2.10}), $\hat{w_1}$ can be represented as 
	\begin{equation*}
		\label{}
		\begin{aligned}
			\hat{w_1}&=&\frac{G(A_1A_2-E^2)-\alpha^{\prime}\xi_1Q(A_1+A_2+2E) }{(A_1A_2-E^2)^2+(A_1+A_2+2E)^2\alpha'^2\xi_1^2|\xi|^2},
		\end{aligned}
	\end{equation*}
	where
	\beno
	Q=\xi\times G+  i\alpha^{\prime}\xi_1(A_1+A_2+2E)\xi \hat{g} . 
	\eeno
	For simplicity, write
	\beno
	\hat{w_1}\left[(A_1A_2-E^2)^2+(A_1+A_2+2E)^2\alpha'^2\xi_1^2|\xi|^2\right]=I_1+I_2+I_3,
	\eeno
	where
	\beno
	I_1&=& G(A_1A_2-E^2)\\
	&=& (A_1A_2-E^2)\left[ A_2\hat{f_1} - E\hat{f_2} + (E-A_2)i\xi\hat{p} + C_1\times(\hat{f_1}+\hat{f_2})   \right],
	\eeno
	\beno
	I_2&=&  -\alpha^{\prime}\xi_1(A_1+A_2+2E)   \xi\times G\\
	&=& -\alpha'\xi_1(A_1+A_2+2E)\left\{A_2\xi\times\hat{f_1} - E\xi\times\hat{f_2} + \alpha'\xi_1\xi\times\left[\xi\times(\hat{f_1}+\hat{f_2})\right]\right\},
	\eeno
	\beno
	I_3&=&  -\alpha'^2\xi_1^2(A_1+A_2+2E)^2 i\xi \hat{g},
	\eeno
	by recalling 
	(\ref{eq:symbel}) and (\ref{eq:G}).
	Furthermore,  using \eqref{eq:pressure}
	, (\ref{eq:symbel}) and (\ref{eq:G}) we have
	\ben
	\label{eq:w1 decomposition}
	&&\hat{w_1}\left\{\left[(\mu_1^2-\mu_2^2)|\xi|^4+\xi_1^2\right]^2 + 4(\mu_1+\mu_2)^2(\alpha^{\prime})^2\xi_1^2|\xi|^6\right\}\nonumber\\
	& = &\left[(\mu_1^2-\mu_2^2)|\xi|^4+\xi_1^2\right](\mu_1 |\xi|^2 - i\xi_1)\left(I - \frac{\xi\otimes\xi}{|\xi|^2}\right)\hat{f_1} \nonumber\\
	&&+\left[(\mu_1^2-\mu_2^2)|\xi|^4+\xi_1^2\right]\mu_2|\xi|^2\left(\frac{\xi\otimes\xi}{|\xi|^2}-I_3\right)\hat{f_2} \nonumber\\
	&&-\alpha^{\prime}\xi_1\left[(\mu_1+\mu_2)|\xi|^2 - i\xi_1\right]\left((\mu_1+\mu_2)|\xi|^2 - i\xi_1\right)\xi\times\hat{f_1}\nonumber\\
	&&+ \alpha^{\prime}\xi_1[(\mu_1+\mu_2)|\xi|^2 - i\xi_1]\left((\mu_1+\mu_2)|\xi|^2 +i\xi_1\right)\xi\times\hat{f_2}  \nonumber\\
	&& -2(\mu_1+\mu_2)(\alpha^{\prime})^2\xi_1^2 |\xi|^2\xi\times[\xi\times(\hat{f_1} + \hat{f_2})] \nonumber\\ 
	&&- \frac{i\xi\hat{g}}{|\xi|^2}\left\{\left[(\mu_1^2-\mu_2^2)|\xi|^4+\xi_1^2\right]^2 + 4(\mu_1+\mu_2)^2(\alpha^{\prime})^2\xi_1^2|\xi|^6\right\}\nonumber\\
	&:=&\Phi_1\hat{f_1}+\Phi_2\hat{f_2}+\Phi_3\hat{f_1} +\Phi_4\hat{f_2}+\Phi_5(\hat{f_1}+\hat{f_2}) +\Phi_6\hat{g}.
	\een
	
	Let 
	\ben\Phi_0 := \left\{\left[(\mu_1^2-\mu_2^2)|\xi|^4+\xi_1^2\right]^2 + 4(\mu_1+\mu_2)^2(\alpha^{\prime})^2\xi_1^2|\xi|^6\right\},\een
	and to prove the assertion (\ref{eq:w1w2 estimate}), we estimate $\Phi_1,\cdots,\Phi_6$, respectively.

	{\bf \underline{Estimates of $\Phi_1$}.} Note that
	$$
	\frac{\Phi_1}{\Phi_0} = \frac{\big[(\mu_1^2-\mu_2^2)|\xi|^4+\xi_1^2\big](\mu_1 |\xi|^2 - i\xi_1)(I - \frac{\xi\otimes\xi}{|\xi|^2})}{\big[(\mu_1^2-\mu_2^2)|\xi|^4+\xi_1^2\big]^2 + 4(\mu_1+\mu_2)^2(\alpha^{\prime})^2\xi_1^2|\xi|^6}.
	$$
	For all $\xi \in \left\{\xi \in \mathbb{R}^3:\left|\xi_i\right|>0\right.$, $\left.i=1, 2, 3\right\}$, Young's inequality yields
	\beno    	
	&&\left|\xi_1\right|^{1 / 2}\left|\xi_2\right|^{1 / 2}\left|\xi_3\right|^{1 / 2}\left|\frac{\Phi_1}{\Phi_0}\right| \nonumber\\
	&\leq& C \left|\frac{\left|\xi_1\right|^{1 / 2}\left|\xi_2\right|^{1 / 2}\left|\xi_3\right|^{1 / 2}}{\mu_1|\xi|^2+i\xi_1}\cdot\frac{\big[(\mu_1^2-\mu_2^2)|\xi|^4+\xi_1^2\big](\mu_1^2|\xi|^4+\xi_1^2)}{\big[(\mu_1^2-\mu_2^2)|\xi|^4+\xi_1^2\big]^2 + 4(\mu_1+\mu_2)^2(\alpha^{\prime})^2\xi_1^2|\xi|^6}\right| \nonumber\\
	&\leq& C\frac{|\xi_1|+\xi_2^2+\xi_3^2}{\sqrt{\xi_1^2 + \mu_1^2|\xi|^4}} \leq C(\mu_1,\mu_2),
	\eeno
	and for $i,j=1,2,3$, there also holds
	\beno
	\left|\xi_i\xi_j\frac{\Phi_1}{\Phi_0}\right| \leq C(\mu_1,\mu_2).
	\eeno
	Moreover,  it is immediately seen that
	\beno
	|\xi_1|^{\hat{\kappa}_1+\frac12}|\xi_2|^{\hat{\kappa}_2+\frac12}|\xi_3|^{\hat{\kappa}_3+\frac12}\left|\frac{\partial^{\kappa'}(\Phi_1/\Phi_0)}{\partial\xi_1^{\hat{\kappa}_1}\partial\xi_2^{\hat{\kappa}_2}\partial\xi_3^{\hat{\kappa}_3}}\right| \leq C(\mu_1,\mu_2),
	\eeno
	and
	\beno
	|\xi_1|^{\hat{\kappa}_1}|\xi_2|^{\hat{\kappa}_2}|\xi_3|^{\hat{\kappa}_3}\left|\frac{\partial^{\kappa'}(\xi_i\xi_j\Phi_1/\Phi_0)}{\partial\xi_1^{\hat{\kappa}_1}\partial\xi_2^{\hat{\kappa}_2}\partial\xi_3^{\hat{\kappa}_3}}\right| \leq C(\mu_1,\mu_2),
	\eeno
	for some $\hat{\kappa}_i$ is zero or one, $\kappa'=\sum_{i=1}^3\hat{\kappa}_i=0,1,2,3.$
	
	The similar arguments also hold for  $\Phi_2,\cdots,\Phi_5.$
	
	{\bf \underline{Estimates of $\Phi_6$}.} Notice that
	\ben\label{eq:decomposition}
	\frac{1}{|\xi|^2} i \xi \hat{g} = \frac{1}{|\xi|^2 + i\xi_1}i\xi\hat{g} - \frac{\xi\xi_1}{|\xi|^2} \frac{1}{|\xi|^2 + i\xi_1}\hat{g}.
	\een
	Writing $$
	\Phi_{61} = \frac{1}{|\xi|^2 + i\xi_1} , \quad
	\Phi_{62} = \frac{-\xi_1\xi}{|\xi|^2(|\xi|^2 + i\xi_1)},
	$$
	then for $\Phi_{61}$ and $\Phi_{62}$, (\ref{eq:bound of phi}) still holds for $\beta = \frac{1}{2}$ or $\beta = 0$. 
	
	Concluding the estimates of $\Phi_1,\cdots,\Phi_6$, using Theorem \ref{thm:lizorkin} we have
	\ben\label{eq:w1 estimate}
	\left\|\nabla^2 w_1\right\|_{L^q\left(\mathbb{R}^3\right)}+\left\|w_1\right\|_{L^r\left(\mathbb{R}^3\right)} \leq C(q, \mu_1, \mu_2)\left(\left\|(f_1,f_2)\right\|_{L^q\left(\mathbb{R}^3\right)}+\|g\|_{W^{1, q}\left(\mathbb{R}^3\right)}\right).
	\een
	
	{\bf \underline{Estimates of pressure.} } 
	For $i\xi \hat{p}$, recall
	\eqref{eq:pressure} and using (\ref{eq:decomposition}) we have
	\ben\label{eq:pressure decom}
	i\xi\hat{p}& =& \frac{1}{(\mu_1 - \mu_2)|\xi|^2 - i\xi_1}\left\{(\mu_1|\xi|^2 - i\xi_1)\frac{\xi\otimes\xi\cdot\hat{f_1}}{|\xi|^2} - \mu_2\xi\otimes\xi\cdot\hat{f_2}\right\}\nonumber\\
	&&+\frac{1}{(\mu_1 - \mu_2)|\xi|^2 - i\xi_1}\left\{\left[(\mu_1^2 - \mu_2^2)|\xi|^4 +\xi_1^2\right]\frac{i\xi\hat{g}}{|\xi|^2}\right\}\nonumber\\
	&=&\Phi_{71}\cdot\hat{f_1}+\Phi_{72}\cdot\hat{f_2}+\Phi_{73}i\xi\hat{g}+\Phi_{74}\hat{g},
	\een
	where
	$$
	\Phi_{71} = \frac{\mu_1|\xi|^2 - i\xi_1}{(\mu_1 - \mu_2)|\xi|^2 - i\xi_1}\frac{\xi\otimes\xi}{|\xi|^2},\quad
	\Phi_{72} = \frac{-\mu_2\xi\otimes\xi}{(\mu_1 - \mu_2)|\xi|^2 - i\xi_1},
	$$
	$$
	\Phi_{73} = \frac{(\mu_1^2 - \mu_2^2)|\xi|^4 +\xi_1^2}{(\mu_1 - \mu_2)|\xi|^2 - i\xi_1}\frac{1}{|\xi|^2 + i\xi_1},\quad
	\Phi_{74} = \frac{(\mu_1^2 - \mu_2^2)|\xi|^4 +\xi_1^2}{(\mu_1 - \mu_2)|\xi|^2 - i\xi_1}\frac{-\xi_1\xi}{|\xi|^2(|\xi|^2 + i\xi_1)},
	$$
	which satisfy (\ref{eq:bound of phi}) with $\beta=0$. Hence, using Theorem \ref{thm:lizorkin} again, we get
	\ben\label{eq:p estimate}
	\|\nabla p\|_{L^q\left(\mathbb{R}^3\right)} 
	\leq C(q, \mu_1, \mu_2)\left(\left\|f_1\right\|_{L^q\left(\mathbb{R}^3\right)}+\left\|f_2\right\|_{L^q\left(\mathbb{R}^3\right)}+\|g\|_{W^{1, q}\left(\mathbb{R}^3\right)}\right).
	\een
	
	{\bf \underline{Estimates of $w_2$.} }  Recalling \eqref{equ2.7} and (\ref{eq:symbel}) we get
	\ben
	\hat{w_2} = -\frac{(\mu_1+\mu_2) |\xi|^2 + i\xi_1}{(\mu_1+\mu_2) |\xi|^2 - i\xi_1}\hat{w_1}+\frac{\hat{f_1}+\hat{f_2}-2i\xi\hat{p}}{(\mu_1+\mu_2) |\xi|^2 - i\xi_1}.
	\een
	Using (\ref{eq:w1 decomposition}), (\ref{eq:decomposition}) and (\ref{eq:pressure decom}), $\hat{w_2}$ is the sum of the following terms multiplying the corresponding $\hat{f_1}, \hat{f_2}, \hat{g}, i\xi\hat{g} $:
	\beno
	&&\frac{(\mu_1+\mu_2) |\xi|^2 + i\xi_1}{(\mu_1+\mu_2) |\xi|^2 - i\xi_1}\frac{\Phi_k}{\Phi_0}, \quad k=1,\cdots,5;\\
	&&\frac{(\mu_1+\mu_2) |\xi|^2 + i\xi_1}{(\mu_1+\mu_2) |\xi|^2 - i\xi_1}{\Phi_{6k}}, \quad k=1,2;\\
	&&  \frac{1}{(\mu_1+\mu_2) |\xi|^2 - i\xi_1}, \quad \frac{1}{(\mu_1+\mu_2) |\xi|^2 - i\xi_1}{\Phi_{7k}}, \quad k=1,2,3,4,
	\eeno
	which satisfy (\ref{eq:bound of phi}) of Theorem \ref{thm:lizorkin} with $\beta = \frac{1}{2}$ or $\beta = 0$. Thus we also get the estimate of $w_2$ in (\ref{eq:w1w2 estimate}).
	The proof is complete.
	
\end{proof}

\begin{Lem}
	\label{lem2.6}
	Assume that $1<q<2$, $f_1, f_2 \in L^q\left(\mathbb{R}^3\right)$, $g \in W^{1, q}\left(\mathbb{R}^3\right)$, $\mathcal{M} \in$ $\left(L^2\left(\mathbb{R}^3\right)\right)^{6 \times 6}$ and $b \in L^{\infty}(\mathbb{R}^3)$. Let $\mu_1, \mu_2$ and $\alpha^{\prime}$ be constants satisfying $\mu_1>0$ and $\mu_1>|\mu_2|$. Moreover, $(w_1, w_2, p)$ satisfies the following perturbation  Oseen system:
	\ben
	\label{eq:perturbation oseen}
	\left\{\begin{aligned}
		-\mu_1\Delta\left(\begin{array}{l}
			w_1 \\
			w_2
		\end{array}\right)+\partial_{x_1}\left(\begin{array}{l}
			w_1 \\
			-w_2
		\end{array}\right)-b \cdot\nabla\left(\begin{array}{l}
			\nabla \times (w_1 - w_2) \\
			\nabla \times (w_2 - w_1)
		\end{array}\right) \\
		\quad \quad\quad -\mu_2\Delta\left(\begin{array}{l}
			w_2 \\
			w_1
		\end{array}\right)
		+\mathcal{M}\left(\begin{array}{l}
			w_1 \\
			w_2
		\end{array}\right)\\-\alpha^{\prime} \partial_{x_1}\left(\begin{array}{l}
			\nabla \times (w_1 - w_2) \\
			\nabla \times (w_2 - w_1)
		\end{array}\right)
		+\left(\begin{array}{c}
			\nabla p \\
			\nabla p
		\end{array}\right)&=\left(\begin{array}{l}
			f_1 \\
			f_2
		\end{array}\right), \\
		\nabla \cdot w_1&=g. \end{aligned}
	\right.
	\een
	There exists a small constant $\varepsilon_0>0$ such that if
	$$
	\|\mathcal{M}\|_{L^2\left(\mathbb{R}^3\right)}+\|b\|_{L^{\infty}(\mathbb{R}^3)} <\varepsilon_0,
	$$
	then there exists a unique $(w_1,w_2, p) \in \left(\dot{W}^{2, q}\left(\mathbb{R}^3\right) \cap L^r\left(\mathbb{R}^3\right)\right)^6 \times$ $\left(\dot{W}^{1, q}\left(\mathbb{R}^3\right) / \mathcal{P}_0\left(\mathbb{R}^3\right)\right)$ solves (\ref{eq:perturbation oseen}) such that
	\ben
	\label{eq:perturbation w1w2}
	&& \left\|\left(\nabla^2 w_1, \nabla^2 w_2\right)\right\|_{L^q\left(\mathbb{R}^3\right)}+\|(w_1, w_2)\|_{L^r\left(\mathbb{R}^3\right)}+\|\nabla p\|_{L^q\left(\mathbb{R}^3\right)}\nonumber \\
	& \leq &C\left(q, \mu_1, \mu_2, \varepsilon_0\right)\left(\left\|f_1\right\|_{L^q\left(\mathbb{R}^3\right)}+\left\|f_2\right\|_{L^q\left(\mathbb{R}^3\right)}+\|g\|_{W^{1, q}\left(\mathbb{R}^3\right)}\right),
	\een
	where $r=\left(\frac{1}{q}-\frac{1}{2}\right)^{-1}$.
\end{Lem}
\begin{proof}
	We denote the Banach spaces
	$$
	\begin{aligned}
		& \mathcal{X}:=\left(\dot{W}^{2, q}\left(\mathbb{R}^3\right) \cap L^r\left(\mathbb{R}^3\right)\right)^6 \times \dot{W}^{1, q}\left(\mathbb{R}^3\right) / \mathcal{P}_0\left(\mathbb{R}^3\right), \\
		& \mathcal{Y}:=\left(L^q\left(\mathbb{R}^3\right)\right)^6 \times W^{1, q}\left(\mathbb{R}^3\right),
	\end{aligned}
	$$
	with norms
	$$
	\begin{aligned}
		\|(v, B, p)\|_{\mathcal{X}} & :=\left\|\left(\nabla^2 v, \nabla^2 B \right)\right\|_{L^q\left(\mathbb{R}^3\right)}+\|(v, B)\|_{L^r\left(\mathbb{R}^3\right)}+\|\nabla p\|_{L^q\left(\mathbb{R}^3\right)}, \\
		\left\|\left(f_1, f_2, g\right)\right\|_{\mathcal{Y}} & :=\left\|f_1\right\|_{L^q\left(\mathbb{R}^3\right)}+\left\|f_2\right\|_{L^q\left(\mathbb{R}^3\right)}+\|g\|_{W^{1, q}\left(\mathbb{R}^3\right)},
	\end{aligned}
	$$
	respectively. Thanks to Lemma \ref{lem:mixed ossen}, the operator $\mathcal{K}$ which is defined by
	$$
	\begin{aligned}
		\mathcal{K}:\qquad \mathcal{X} & \mapsto \mathcal{Y} \\
		\left(\begin{array}{l}
			w_1 \\
			w_2 \\
			p
		\end{array}\right) & \mapsto\left(\begin{array}{l}-\mu_1 \Delta w_1+\partial_{x_1} w_1-\mu_2\Delta w_2+\nabla p- \alpha^{\prime}\partial_{x_1}(\nabla\times(w_1-w_2))\\
			-\mu_1 \Delta w_2-\partial_{x_1} w_2-\mu_2\Delta w_1+\nabla p- \alpha^{\prime}\partial_{x_1}(\nabla\times(w_2-w_1))\\
			\nabla \cdot w_1\end{array}\right),
	\end{aligned}
	$$
	admits a bounded inverse operator $\mathcal{K}^{-1}$ and there exists a constant $C_2=C(q, \mu_1, \mu_2)$ such that
	$$
	\left\|\left(\nabla^2 w_1, \nabla^2 w_2\right)\right\|_{L^q\left(\mathbb{R}^3\right)}+\|(w_1, w_2)\|_{L^r\left(\mathbb{R}^3\right)}+\|\nabla p\|_{L^q\left(\mathbb{R}^3\right)} \leq C_2\|\mathcal{K}(w_1, w_2, p)\|_{\mathcal{Y}} .
	$$
	Note that the operator $\mathcal{H}: \mathcal{X} \mapsto \mathcal{Y}$, which is defined by
	$$
	\mathcal{H}(w_1, w_2, p)=\left(\left[\mathcal{M}\left(\begin{array}{l}
		w_1 \\
		w_2
	\end{array}\right)\right], 0\right) -\left(\left[b \cdot \nabla\left(\begin{array}{l}
		\nabla \times (w_1-w_2) \\
		\nabla \times (w_2-w_1)
	\end{array}\right)\right], 0\right)
	$$
	satisfies
	$$
	\begin{aligned}
		\|\mathcal{H}(w_1, w_2, p)\| _{\mathcal{Y}} \leq C_2(\|\mathcal{M}\|_{L^2\left(\mathbb{R}^3\right)}+\|b\|_{L^{\infty}(\mathbb{R}^3)})\|\mathcal{K}(w_1, w_2, p)\|_{\mathcal{Y}}.
	\end{aligned}
	$$
	Choosing $\varepsilon_0 = (C_2+1)^{-1}$, it follows that
	$$
	C_2(\|\mathcal{M}\|_{L^2\left(\mathbb{R}^3\right)}+\|b\|_{L^{\infty}(\mathbb{R}^3)})<1 .
	$$
	Then Theorem \ref{lem2.5} yields if $\|\mathcal{M}\|_{L^2\left(\mathbb{R}^3\right)}+\|b\|_{L^{\infty}(\mathbb{R}^3)} <\varepsilon_0$, the operator $\mathcal{K}+\mathcal{H}$ admits a bounded inverse. Then, the estimate (\ref{eq:perturbation w1w2}) holds automatically by the formula (\ref{eq:kato}).
\end{proof}

\label{sec3}
\begin{proof}[Proof of Theorem \ref{the1.1}.]
	Without loss of generality, let $\alpha = 1$. Denoting $\eta:=B-B_{\infty}$, one derives
	\begin{equation}
		\label{equ3.1}
		\left\{\begin{array}{l}
			u \cdot \nabla u+\partial_{x_1} \eta+\nabla p-\eta \cdot \nabla \eta-\kappa \Delta u=0, \\
			u \cdot \nabla \eta+\partial_{x_1} u-\eta \cdot \nabla u-\nu \Delta \eta=\nabla\times((\nabla\times B)\times B), \\
			\nabla \cdot u=0, \quad \nabla \cdot \eta=0.
		\end{array} \right.
	\end{equation}
	As in \cite{Wei}, we introduce the Els\"{a}sser variables:
	$$
	w_1=u+\eta ; \quad w_2=u-\eta.
	$$
	Denote $\mu_1 = \frac{1}{2}(\kappa + \nu)$, $\mu_2 = \frac{1}{2}(\kappa - \nu)$, then $w_1$ and $w_2$ satisfy the following system
	\begin{equation}
		\label{equ3.2}
		\left\{\begin{array}{l}
			w_2 \cdot \nabla w_1+\partial_{x_1} w_1+\nabla p-\mu_1\Delta w_1-\mu_2\Delta w_2=\nabla\times((\nabla\times B)\times B), \\
			w_1 \cdot \nabla w_2-\partial_{x_1} w_2+\nabla p-\mu_1\Delta w_2-\mu_2\Delta w_1=-\nabla\times((\nabla\times B)\times B), \\
			\nabla \cdot w_1=0, \quad \nabla \cdot w_2=0, \\
			\int_{\mathbb{R}^3}\left|\nabla w_1\right|^2 d x<+\infty, \quad \int_{\mathbb{R}^3}\left|\nabla w_2\right|^2 d x<+\infty.
		\end{array} \right.
	\end{equation}
	Next we will prove $u=\eta\equiv0$, which is divided into three steps.
	
	{\bf Step I: the estimate of pressure.}  Acting the divergence operator on $(\ref{equ3.1})_1$, we have
	$$
	\begin{aligned}
		p =\sum_{i, j=1}^3-\frac{1}{\triangle} \partial_i \partial_j\left(u_i u_j-\eta_i \eta_j\right) =\sum_{i, j=1}^3 \mathcal{R}_i \mathcal{R}_j\left(u_i u_j-\eta_i \eta_j\right),
	\end{aligned}
	$$
	where $\mathcal{R}_i=\frac{\partial_i}{\sqrt{-\Delta}}$ denotes the i-th Riesz transform. By the continuity of the operator $\mathcal{R}_i$ on Lebesgue space $L^\ell\left(\mathbb{R}^3\right)$ for $2<\ell<+\infty$, we have
	$$
	\|p\|_{L^{\frac{\ell}{2}}(\mathbb{R}^3)} \leq C\|(u, \eta)\|_{L^\ell (\mathbb{R}^3)}^2.
	$$
	%which implies that if $u$ and $\eta$ belong to the space $L^r\left(\mathbb{R}^3\right)$, then the pressure $p$ belongs to the space $L^{\frac{r}{2}}\left(\mathbb{R}^3\right)$.
	Due to the Sobolev imbedding,
	$$
	\|u\|_{L^6\left(\mathbb{R}^3\right)}+\|\eta\|_{L^6\left(\mathbb{R}^3\right)} \leq C\left(\|\nabla u\|_{L^2\left(\mathbb{R}^3\right)}+\|\nabla \eta\|_{L^2\left(\mathbb{R}^3\right)}\right)<\infty,
	$$
	$$
	\|w_1\|_{L^6\left(\mathbb{R}^3\right)}+\|w_2\|_{L^6\left(\mathbb{R}^3\right)} \leq C\left(\|\nabla w_1\|_{L^2\left(\mathbb{R}^3\right)}+\|\nabla w_2\|_{L^2\left(\mathbb{R}^3\right)}\right)<\infty,
	$$
	and we have 
	$$
	\|u\|_{L^6\left(\mathbb{R}^3\right)}+\|\eta\|_{L^6\left(\mathbb{R}^3\right)}+\|p\|_{L^3(\mathbb{R}^3)}< \infty.
	$$

	{\bf Step II: $L^3$ estimate of $w_1$ and  $w_2$.}   Define a cut-off function $\psi \in C^{\infty}\left(\mathbb{R}^3\right)$ such that
	\begin{equation}
		\label{eq:psi}
		\psi(x)=\psi(|x|)=\left\{\begin{array}{lll}
			1, & \text { if } & |x|>2 M , \\
			0, & \text { if } & |x|<M,
		\end{array}\right.
	\end{equation}
	and $0 \leq \psi(x) \leq 1, x \in B_{2 M}\setminus B_M$.
	Multiplying $\psi$ on the both sides of $(\ref{equ3.2})_{1,2,3}$, it follows that
	\ben
	\label{equ3.4}
	\left\{\begin{aligned} 
		\partial_{x_1}\left(\begin{array}{c}
			\psi w_1 \\
			-\psi w_2
		\end{array}\right)+\left(\begin{array}{cc}
			0 & \left(\nabla w_1\right)^T \chi_{B_M^c} \\
			\left(\nabla w_2\right)^T \chi_{B_M^c} & 0
		\end{array}\right) \cdot\left(\begin{array}{l}
			\psi w_1 \\
			\psi w_2
		\end{array}\right)\\-\mu_1\Delta\left(\begin{array}{c}
			\psi w_1 \\
			\psi w_2
		\end{array}\right)
		+\left(\begin{array}{c}
			\nabla(\psi p) \\
			\nabla(\psi p)
		\end{array}\right) 
		-\mu_2\Delta\left(\begin{array}{c}
			\psi w_2 \\
			\psi w_1
		\end{array}\right)\\-\frac{1}{2}\chi_{B_M^c}\eta\cdot\nabla\left(\begin{array}{c}
			\nabla\times(\psi w_1 - \psi w_2) \\
			\nabla\times(\psi w_2 - \psi w_1) 
		\end{array}\right) \\
		+\frac{1}{2}\partial_{x_1}\left(\begin{array}{c}
			\nabla\times(\psi w_1 - \psi w_2) \\
			\nabla\times(\psi w_2 - \psi w_1) 
		\end{array}\right)
		&=\left(\begin{array}{c}
			G_1(\psi) \\
			G_2(\psi)
		\end{array}\right),  \\
		\nabla \cdot\left(\psi w_1\right)&=\nabla \psi \cdot w_1,
	\end{aligned}\right.
	\een
	where
	$$
	\begin{aligned}
		G_1(\psi)&=-2\mu_1 \nabla \psi \cdot \nabla w_1-\mu_1(\Delta \psi) w_1+w_1 \partial_{x_1} \psi+(\nabla \psi) p \\
		&\quad-2\mu_2 \nabla \psi \cdot \nabla w_2-\mu_2(\Delta \psi) w_2 +\partial_{x_1}\psi\nabla\times\eta+\partial_{x_1}(\nabla\psi\times\eta) \\
		&\quad-(\eta\cdot\nabla\psi)\nabla\times\eta - \eta\cdot\nabla(\nabla\psi\times\eta) - \psi(\nabla\times B)\cdot\nabla B
	\end{aligned}
	$$
	and $$
	\begin{aligned}
		G_2(\psi)&=-2\mu_1 \nabla \psi \cdot \nabla w_2-\mu_1(\Delta \psi) w_2-w_2 \partial_{x_1} \psi+(\nabla \psi) p \\
		&\quad-2\mu_2 \nabla \psi \cdot \nabla w_1-\mu_2(\Delta \psi) w_1 -\partial_{x_1}\psi\nabla\times\eta - \partial_{x_1}(\nabla\psi\times\eta) \\
		&\quad+(\eta\cdot\nabla\psi)\nabla\times\eta + \eta\cdot\nabla(\nabla\psi\times\eta) + \psi(\nabla\times B)\cdot\nabla B.
	\end{aligned}
	$$
	Note that 
	\begin{equation}
		\label{equ3.5}
		\begin{aligned}
			\left\|G_1(\psi)\right\|_{L^{\frac{6}{5}}\left(\mathbb{R}^3\right)} \leq & C(\mu_1, \mu_2)\left[\|\nabla \psi\|_{L^3\left(\mathbb{R}^3\right)}\|\nabla w_1\|_{L^2\left(\mathbb{R}^3\right)}+\|\Delta \psi\|_{L^{\frac{3}{2}}\left(\mathbb{R}^3\right)}\|w_1\|_{L^6\left(\mathbb{R}^3\right)} \right]\\
			& +C(\mu_1, \mu_2)\left[\left\|\partial_{x_1} \psi\right\|_{L^{\frac{3}{2}}\left(\mathbb{R}^3\right)}\|w_1\|_{L^6\left(\mathbb{R}^3\right)}+\|\nabla \psi\|_{L^2\left(\mathbb{R}^3\right)}\|p\|_{L^3\left(\mathbb{R}^3\right)} \right]\\
			& +C(\mu_1, \mu_2)\left[\|\nabla \psi\|_{L^3\left(\mathbb{R}^3\right)}\|\nabla w_2\|_{L^2\left(\mathbb{R}^3\right)}+\|\Delta \psi\|_{L^{\frac{3}{2}}\left(\mathbb{R}^3\right)}\|w_2\|_{L^6\left(\mathbb{R}^3\right)} \right]\\
			& +C(\mu_1, \mu_2)\left[\|\nabla\psi\|_{L^3 (\mathbb{R}^3)}\|\nabla \eta\|_{L^2 (\mathbb{R}^3)} +
			\|\Delta\psi\|_{L^{\frac{3}{2}} (\mathbb{R}^3)}\|\eta\|_{L^6 (\mathbb{R}^3)} \right]\\
			&+C(\mu_1, \mu_2)\left[\|\nabla\psi\|_{L^6 (\mathbb{R}^3)}\|\nabla \eta\|_{L^2 (\mathbb{R}^3)}\| \eta\|_{L^6 (\mathbb{R}^3)} + \|\Delta \psi\|_{L^2 (\mathbb{R}^3)}\| \eta\|^2 _{L^6 (\mathbb{R}^3)} \right]\\
			&+ C(\mu_1, \mu_2)\|\psi\|_{L^{\infty}(\mathbb{R}^3)}\|\nabla B\|^{\frac{1}{3}}_{L^{\infty}(\mathbb{R}^3)}\|\nabla B\|^{\frac{5}{3}}_{L^{2}(\mathbb{R}^3)}<C(\mu_1, \mu_2, M),
		\end{aligned}
	\end{equation}
	and similarly 
	$$
	\left\|G_2(\psi)\right\|_{L^{\frac{6}{5}}\left(\mathbb{R}^3\right)}\leq C(\mu_1, \mu_2, M).
	$$
	Also,
	$$
	\begin{aligned}
		\|\nabla \psi \cdot w_1\|_{W^{1, \frac{6}{5}}\left(\mathbb{R}^3\right)} \leq & \|\nabla \psi\|_{L^{\frac{3}{2}}\left(\mathbb{R}^3\right)}\|w_1\|_{L^6\left(\mathbb{R}^3\right)}+\left\|\nabla^2 \psi\right\|_{L^{\frac{3}{2}}\left(\mathbb{R}^3\right)}\|w_1\|_{L^6\left(\mathbb{R}^3)\right.} \\
		& +\|\nabla \psi\|_{L^3\left(\mathbb{R}^3\right)}\|\nabla w_1\|_{L^2\left(\mathbb{R}^3\right)} \\
		\leq & C(M)<\infty .
	\end{aligned}
	$$
	Let $\varepsilon_0$ as in Lemma \ref{lem2.6}, by the finite Dirichlet energy and the uniform convergence of $\eta$ at infinity, there exists an $M>0$ such that
	$$
	\|\nabla w_1\|_{L^2 (B_M^c)} +\|\nabla w_2\|_{L^2 (B_M^c)} +\|\frac{1}{2} \eta\|_{L^{\infty}(B_M^c)}<\varepsilon_0.
	$$
	Then applying Lemma \ref{lem2.6} to the system (\ref{equ3.4}), choosing $q=\frac{6}{5}$ and $\alpha^{\prime} = -\frac{1}{2}$, we have $\psi w_1, \psi w_2 \in L^3\left(\mathbb{R}^3\right)$. 
	It follows that
	\begin{equation}
		\label{equ3.6}
		\|u\|_{L^3\left(\mathbb{R}^3\right)}+\|\eta\|_{L^3\left(\mathbb{R}^3\right)} + \|p\|_{L^{\frac{3}{2}}\left(\mathbb{R}^3\right)}<\infty,
	\end{equation}
	since $w_1$ and $w_2$  are smooth.
	
	{\bf Step III: vanishing of $u$ and $\eta$.}    
	Choose a cut-off function $\varphi \in C_c^{\infty}\left(\mathbb{R}^3\right)$ such that
	$$
	\varphi(x)=\varphi(|x|)=\left\{\begin{array}{lll}
		1, & \text { if } & |x|<1 ; \\
		0, & \text { if } & |x|>2 ,
	\end{array}\right.
	$$
	and $0 \leq \varphi(x) \leq 1$ for any $1 \leq|x| \leq 2$. Then for any $R>0$, denote
	\ben
	\varphi_R(x):=\varphi\left(\frac{|x|}{R}\right) .
	\een
	Multiplying $(\ref{equ3.1})_1$ by $u \varphi_R$,  $(\ref{equ3.1})_2$ by $\eta \varphi_R$ and integrating over $\mathbb{R}^3$, we have
	\beno
	&&\int_{\mathbb{R}^3} \varphi_R\left(\kappa|\nabla u|^2+\nu|\nabla \eta|^2\right) d x\nonumber \\
	& =&\frac{1}{2} \int_{\mathbb{R}^3} \Delta \varphi_R\left(\kappa|u|^2+\nu|\eta|^2\right) d x+\int_{\mathbb{R}^3}\left(\eta\cdot u\right) \partial_{x_1} \varphi_R d x \nonumber\\
	& &+\int_{\mathbb{R}^3} p u \cdot \nabla \varphi_R d x+\frac{1}{2} \int_{\mathbb{R}^3}\left[\left(|u|^2+|\eta|^2\right) u-2(\eta \cdot u) \eta\right] \cdot \nabla \varphi_R d x\nonumber\\
	&& -\int_{\mathbb{R}^3} (\nabla\times \eta)\times \eta\cdot(\eta\times \nabla \varphi_R)dx
	-\int_{\mathbb{R}^3} (\nabla\times \eta)\times B_{\infty}\cdot(\eta\times \nabla \varphi_R)dx\nonumber\\
	&:=&\sum_{i=1}^6 I_i,
	\eeno
	where
	we used
	$$
	\begin{aligned}
		\int \nabla \times F \cdot G \phi d x & =\int \nabla \times G \cdot F \phi d x+\int G \times F \cdot \nabla \phi d x \\
		& =\int \nabla \times G \cdot F \phi d x-\int F \cdot(G \times \nabla \phi) d x ,
	\end{aligned}
	$$
	where $F, G \in \mathbb{R}^3$ and $\phi$ is a cut-off function.
	
	By H${\rm \ddot{o}}$lder's inequality, we have
	$$
	\begin{aligned}
		\left|I_1\right| 
		& \lesssim R^{-1}\left(\|u\|_{L^3\left(B_{2 R}\setminus B_R\right)}^2+\|\eta\|_{L^3\left(B_{2 R}\setminus B_R\right)}^2\right) , \\
		\left|I_2\right|
		& \lesssim\|u\|_{L^3\left(B_{2 R}\setminus B_R\right)}\|\eta\|_{L^3\left(B_{2 R}\setminus B_R\right)} , \\
		\left|I_3\right| 
		& \lesssim R^{-1}\|p\|_{L^{\frac{3}{2}}\left(B_{2 R}\setminus B_R\right)}\|u\|_{L^3\left(B_{2 R}\setminus B_R\right)} , \\
		\left|I_4\right| 
		& \lesssim R^{-1}\left(\|u\|_{L^3\left(B_{2 R}\setminus B_R\right)}^3+\|\eta\|_{L^3\left(B_{2 R}\setminus B_R\right)}^3\right), \\
		\left|I_5\right| 
		& \lesssim R^{-\frac{1}{2}}\|\eta\|^2 _{L^6\left(B_{2 R}\setminus B_R\right)}\|\nabla\eta\|_{L^2\left(B_{2 R}\setminus B_R\right)}, \\
		\left|I_6\right| 
		& \lesssim \|\eta\|_{L^6\left(B_{2 R}\setminus B_R\right)}\|\nabla\eta\|_{L^2\left(B_{2 R}\setminus B_R\right)}.
	\end{aligned}
	$$
	Hence, letting $R \rightarrow \infty$ we get
	$$
	\int_{\mathbb{R}^3}\left(\kappa|\nabla u|^2+\nu|\nabla \eta|^2\right) dx=0.
	$$
	It follows that $u$ and $\eta$ are constants. By $u, \eta \in L^3\left(\mathbb{R}^3\right)$, we have $u=\eta \equiv 0$ for all $x \in \mathbb{R}^3$. The proof is complete.
\end{proof}

\section{Hall-MHD system with $u_\infty\neq0$: Proof of Theorem \ref{the1.4}}
\label{sec4}

Next consider the 3D degenerate Oseen system in   $\mathbb{R}^3$:
\begin{equation}
	\label{eq:degenerate linear oseen}
	\left\{\begin{array}{l}
		- \Delta_\kappa v+\partial_{x_1} v+\nabla p=f_1, \\
		- \Delta_\nu B+\partial_{x_1} B=f_2, \\
		\nabla \cdot v=g,
	\end{array}\right.
\end{equation}
where
$\kappa_1, \nu_1\geq 0$, $\kappa_2,\kappa_3,\nu_2,\nu_3>0$.

\begin{Lem}[$L^q$ estimate for degenerate Oseen system]
	\label{lem2.3}
	Let $f_1, f_2 \in L^q\left(\mathbb{R}^3\right), g \in W^{1, q}\left(\mathbb{R}^3\right)$ with $1<q<2$ and $r=\left(\frac{1}{q}-\frac{1}{2}\right)^{-1}$.  Then we have the following results: \\
	(i) when $\kappa_1\nu_1= 0$, there exists a unique solution  $$(v, B, p) \in\left(L^r\left(\mathbb{R}^3\right)\right)^6 \times\left(\dot{W}^{1, q}\left(\mathbb{R}^3\right) / \mathcal{P}_0\left(\mathbb{R}^3\right)\right)$$ of the system  \eqref{eq:degenerate linear oseen} such that
	$$
	\begin{aligned}
		&\quad\|(v, B)\|_{L^r\left(\mathbb{R}^3\right)}+\|\nabla p\|_{L^q\left(\mathbb{R}^3\right)} \\
		& \quad \leq C(q, \kappa_1, \kappa_2, \kappa_3, \nu_2, \nu_3)\left(\left\|f_1\right\|_{L^q\left(\mathbb{R}^3\right)}+\left\|f_2\right\|_{L^q\left(\mathbb{R}^3\right)}+\|g\|_{W^{1, q}\left(\mathbb{R}^3\right)}\right) .
	\end{aligned}
	$$
	(ii) when $\kappa_1\nu_1\neq 0$, there exists a unique solution 
	$$(v, B, p) \in\left(\dot{W}^{2, q}\left(\mathbb{R}^3\right) \cap L^r\left(\mathbb{R}^3\right)\right)^6 \times\left(\dot{W}^{1, q}\left(\mathbb{R}^3\right) / \mathcal{P}_0\left(\mathbb{R}^3\right)\right)$$
	of the system  \eqref{eq:degenerate linear oseen} such that
	$$
	\begin{aligned}
		& \left\|\left(\nabla^2 v, \nabla^2 B\right)\right\|_{L^q\left(\mathbb{R}^3\right)}+\|(v, B)\|_{L^r\left(\mathbb{R}^3\right)}+\|\nabla p\|_{L^q\left(\mathbb{R}^3\right)} \\
		& \quad \leq C(q, \kappa_1, \kappa_2, \kappa_3, \nu_1, \nu_2, \nu_3)\left(\left\|f_1\right\|_{L^q\left(\mathbb{R}^3\right)}+\left\|f_2\right\|_{L^q\left(\mathbb{R}^3\right)}+\|g\|_{W^{1, q}\left(\mathbb{R}^3\right)}\right) .
	\end{aligned}
	$$
\end{Lem}	
\begin{proof}
	Applying the Fourier transform on the system of  \eqref{eq:degenerate linear oseen}, we obtain that
	$$
	\left\{\begin{array}{l}
		(\kappa_1\xi_1^2 + \kappa_2\xi_2^2 + \kappa_3\xi_3^2 +i \xi_1) \hat{v}(\xi)+i \xi \hat{p}(\xi)=\hat{f}_1(\xi), \\
		(\nu_1\xi_1^2 + \nu_2\xi_2^2 + \nu_3\xi_3^2+i \xi_1 ) \hat{B}(\xi)=\hat{f}_2(\xi), \\
		i \xi \cdot \hat{v}=\hat{g}(\xi).
	\end{array} \right. 
	$$
	Direct calculation shows
	$$
	\begin{aligned}
		& \hat{v}(\xi)=\left(\kappa_1\xi_1^2 + \kappa_2\xi_2^2 + \kappa_3\xi_3^2+i \xi_1\right)^{-1}\left(I_3-\frac{\xi \otimes \xi}{|\xi|^2}\right) \cdot \hat{f}_1(\xi)-\frac{i\xi\hat{g}(\xi)}{|\xi|^2}, \\
		& \hat{B}(\xi)=\left(\nu_1\xi_1^2 + \nu_2\xi_2^2 + \nu_3\xi_3^2+i \xi_1\right)^{-1} \hat{f}_2(\xi), \\
		& \hat{p}(\xi)=-\frac{i \xi \cdot \hat{f}_1(\xi)}{|\xi|^2}+\frac{i \xi_1 \hat{g}(\xi)}{|\xi|^2} + \frac{\kappa_1\xi_1^2 + \kappa_2\xi_2^2 + \kappa_3\xi_3^2}{|\xi|^2} \hat{g}(\xi) .
	\end{aligned}
	$$
	
	(i)	For $\kappa_1\nu_1=0$, first we consider $\Phi_\kappa(\xi):=\left(\kappa_1\xi_1^2 + \kappa_2\xi_2^2 + \kappa_3\xi_3^2+i \xi_1\right)^{-1}$ and $\Phi_\nu(\xi):=\left(\nu_1\xi_1^2 + \nu_2\xi_2^2 + \nu_3\xi_3^2+i \xi_1\right)^{-1}$, which satisfy the conditions in Theorem \ref{thm:lizorkin} with $\beta=\frac{1}{2}$. In fact, 
	\beno
	\left|\xi_1\right|^{1 / 2}\left|\xi_2\right|^{1 / 2}\left|\xi_3\right|^{1 / 2}|\Phi_\kappa(\xi)|\leq \frac{\left(\left|\xi_1\right|+\xi_2^2+\xi_3^2\right)}{\sqrt{\xi_1^2+(\kappa_2\xi_2^2+\kappa_3\xi_3^2)^2}} \leq C (\kappa_2, \kappa_3),
	\eeno
	and $\Phi_\nu(\xi)$ is similar. Second, for the term of  $-\frac{i\xi\hat{g}(\xi)}{|\xi|^2}$, it is similar as the estimates in  \eqref{eq:decomposition}. Hence, by Theorem \ref{thm:lizorkin}  we get
	\ben\label{eq:v,h}
	\|(v,B)\|_{L^r\left(\mathbb{R}^3\right)} \leq C(q, \kappa_2, \kappa_3,\nu_2,\nu_3)\left(\left\|f_1\right\|_{L^q\left(\mathbb{R}^3\right)}+\left\|f_2\right\|_{L^q\left(\mathbb{R}^3\right)}+\|g\|_{W^{1, q}\left(\mathbb{R}^3\right)}\right) .\een
	Finally, for the pressure $i\xi\hat{p}$,  we get
	\beno
	i\xi\hat{p}(\xi)=\frac{\xi\otimes \xi \cdot \hat{f}_1(\xi)}{|\xi|^2}-\frac{ \xi_1\xi \hat{g}(\xi)}{|\xi|^2} + \frac{\kappa_1\xi_1^2 + \kappa_2\xi_2^2 + \kappa_3\xi_3^2}{|\xi|^2} i\xi\hat{g}(\xi), 
	\eeno
	which implies
	\ben\label{eq:p only}
	\|\nabla p\|_{L^q\left(\mathbb{R}^3\right)} \leq C(q, \kappa_1, \kappa_2, \kappa_3)\left(\left\|f_1\right\|_{L^q\left(\mathbb{R}^3\right)}+\left\|f_2\right\|_{L^q\left(\mathbb{R}^3\right)}+\|g\|_{W^{1, q}\left(\mathbb{R}^3\right)}\right) .\een

	%When $\kappa_1\neq 0$, $\Phi_\kappa(\xi)$ has been proved in \cite{li2021liouville} and so we set $\kappa_1=0$ in the below. For all $\xi \in \{\xi \in \mathbb{R}^3:\left|\xi_i\right|>0, i=1,2,3\}$, Young's inequality yields
	%			$$
	%			\left|\xi_1\right|^{1 / 2}\left|\xi_2\right|^{1 / 2}\left|\xi_3\right|^{1 / 2}\left|\Phi_\kappa(\xi)\right| \leq \frac{C\left(\left|\xi_1\right|+\xi_2^2+\xi_3^2\right)}{\sqrt{\xi_1^2+(\kappa_2\xi_2^2+\kappa_3\xi_3^2)^2}} \leq C .
	%			$$
	%			This proves (\ref{eq:bound of phi}) for $k_1=k_2=k_3=0$. The proof for non-zero $\left(k_1, k_2, k_3\right)$ is similar.
	
	(ii) For $
	\kappa_1\nu_1\neq0$, it's similar as in \cite{LP}. In fact, the estimates of \eqref{eq:v,h} and \eqref{eq:p only} still hold as in Step I. Moreover, 
	$\xi_i\xi_j\Phi_\kappa(\xi)$ and $\xi_i\xi_j\Phi_\nu(\xi)$ satisfy the condition  in Theorem \ref{thm:lizorkin} with $\beta=0$. The proof is complete. 
\end{proof}

\begin{Lem}
	\label{lem2.7}
	%		Assume $1<q<2, f_1, f_2 \in L^q\left(\mathbb{R}^3\right), g \in W^{1, q}\left(\mathbb{R}^3\right), \mathcal{M} \in$ $\left(L^2\left(\mathbb{R}^3\right)\right)^{6 \times 6}$ and $b \in L^{\infty}(\mathbb{R}^3)$. 
	Assume the same conditions of $q, f_1, f_2, g, \mathcal{M}$ and $b$ in Lemma \ref{lem2.6}. Let $\kappa,\nu>0$, and  $(v, B, p)$ satisfies the generalized Oseen system:
	\begin{equation}
		\label{equ2.16}
		\left\{\begin{aligned}
			-\Delta\left(\begin{array}{l}
				\kappa v \\
				\nu B
			\end{array}\right)+\partial_{x_1}\left(\begin{array}{l}
				v \\
				B
			\end{array}\right)+\mathcal{M}\left(\begin{array}{l}
				v \\
				B
			\end{array}\right)+\left(\begin{array}{c}
				\nabla p \\
				0
			\end{array}\right)+\left(\begin{array}{c}
				0 \\
				b \cdot \nabla(\nabla \times B)
			\end{array}\right)&=\left(\begin{array}{l}
				f_1 \\
				f_2
			\end{array}\right),\\
			\nabla \cdot v&=g.
		\end{aligned}\right. 
	\end{equation}
	Then there exists a small constant $\varepsilon_0 >0$ such that if
	$$	\|\mathcal{M}\|_{L^2\left(\mathbb{R}^3\right)}+\|b\|_{L^{\infty}(\mathbb{R}^3)}<\varepsilon_0 ,
	$$
	then there exists a unique $(v, B, p) \in \left(\dot{W}^{2, q}\left(\mathbb{R}^3\right) \cap L^r\left(\mathbb{R}^3\right)\right)^6 \times$ $\left(\dot{W}^{1, q}\left(\mathbb{R}^3\right) / \mathcal{P}_0\left(\mathbb{R}^3\right)\right)$ solves $(\ref{equ2.16})$ such that
	\begin{equation}
		\label{equ2.17}
		\begin{aligned}
			& \left\|\left(\nabla^2 v, \nabla^2 B\right)\right\|_{L^q\left(\mathbb{R}^3\right)}+\|(v, B)\|_{L^r\left(\mathbb{R}^3\right)}+\|\nabla p\|_{L^q\left(\mathbb{R}^3\right)} \\
			& \quad \leq C\left(\kappa, \nu, q, \varepsilon_0\right)\left(\left\|f_1\right\|_{L^q\left(\mathbb{R}^3\right)}+\left\|f_2\right\|_{L^q\left(\mathbb{R}^3\right)}+\|g\|_{W^{1, q}\left(\mathbb{R}^3\right)}\right),
		\end{aligned}
	\end{equation}
	where $r=\left(\frac{1}{q}-\frac{1}{2}\right)^{-1}$.
\end{Lem}
\begin{proof}
	Thanks to Lemma \ref{lem2.3}(taking $\kappa = \kappa_1 = \kappa_2 = \kappa_3$ and $\nu = \nu_1 = \nu_2 = \nu_3$), the proof of Lemma \ref{lem2.7} is similar to the proof of Lemma \ref{lem2.6} and we omit it. 
\end{proof}

\begin{proof}[Proof of Theorem \ref{the1.4}.]
	Without loss of generality, let $\alpha = 1$. Denoting $v:=u-u_{\infty}$, it follows that
	\begin{equation}
		\label{equ3.7}
		\left\{\begin{array}{l}
			v \cdot \nabla v+\partial_{x_1} v+\nabla p-B \cdot \nabla B-\kappa\Delta v=0, \\
			v \cdot \nabla B+\partial_{x_1} B-B \cdot \nabla v-\nu\Delta B=\nabla \times((\nabla \times B) \times B), \\
			\nabla \cdot v=0, \quad \nabla \cdot B=0.
		\end{array} \right.
	\end{equation}
	In the same way as in the proof of Theorem \ref{the1.1}, we get
	$$
	\|v\|_{L^6\left(\mathbb{R}^3\right)}+\|B\|_{L^6\left(\mathbb{R}^3\right)}+\|p\|_{L^3(\mathbb{R}^3)}< \infty.
	$$
	As Theorem \ref{the1.1}, we will estimate $L^3$ norm of $v$ and $B$ and prove they are vanishing.
	
	{\bf Step I. $L^3$ norm of $v$ and $B$.} Multiplying $\psi$, defined  in (\ref{eq:psi}), on the both sides of (\ref{equ3.7}), it follows that
	\ben
	\label{equ3.8}
	\left\{ \begin{aligned}
		\partial_{x_1}\left(\begin{array}{c}
			\psi v \\
			\psi B
		\end{array}\right)+\left(\begin{array}{c}
			(\nabla v)^T \chi_{B_M^c},-(\nabla B)^T \chi_{B_M^c} \\
			(\nabla B)^T \chi_{B_M^c}^c,-(\nabla v)^T \chi_{B_M^c}
		\end{array}\right) \cdot\left(\begin{array}{c}
			\psi v \\
			\psi B
		\end{array}\right)\\-\Delta\left(\begin{array}{c}
			\kappa\psi v \\
			\nu\psi B
		\end{array}\right)+\left(\begin{array}{c}
			\nabla(\psi p) \\
			0
		\end{array}\right) 
		\\+\left(\begin{array}{c}
			0\\
			(-B) \chi_{B_M^c} \cdot \nabla(\nabla \times \psi B)
		\end{array}\right)&=\left(\begin{array}{c}
			F_1(\psi) \\
			F_2(\psi)
		\end{array}\right),  \\
		\nabla \cdot(\psi v)&=\nabla \psi \cdot v,
	\end{aligned} \right. 
	\een
	where
	$$
	\begin{aligned}
		& F_1(\psi)=-2\kappa \nabla \psi \cdot \nabla v-\kappa(\Delta \psi) v+v \partial_{x_1} \psi+(\nabla \psi) p, \\
		& F_2(\psi)=-2\nu \nabla \psi \cdot \nabla B-\nu(\Delta \psi) B+B\partial_{x_1} \psi - (B\cdot\nabla)\psi(\nabla\times B) \\
		& \qquad \qquad - B\cdot\nabla(\nabla\psi\times B)-\psi(\nabla\times B)\cdot\nabla B.
	\end{aligned}
	$$
	Note that
	$$
	\begin{aligned}
		\left\|F_1(\psi)\right\|_{L^{\frac{6}{5}}\left(\mathbb{R}^3\right)} \leq & C(\kappa)\left(\|\nabla \psi\|_{L^3\left(\mathbb{R}^3\right)}\|\nabla v\|_{L^2\left(\mathbb{R}^3\right)}+\|\Delta \psi\|_{L^{\frac{3}{2}}\left(\mathbb{R}^3\right)}\|v\|_{L^6\left(\mathbb{R}^3\right)}\right. \\
		& \left.+\left\|\partial_{x_1} \psi\right\|_{L^{\frac{3}{2}}\left(\mathbb{R}^3\right)}\|v\|_{L^6\left(\mathbb{R}^3\right)}+\|\nabla \psi\|_{L^2\left(\mathbb{R}^3\right)}\|p\|_{L^3\left(\mathbb{R}^3\right)}\right) \\
		\leq & C(\kappa, M)<\infty,
	\end{aligned}
	$$
	\begin{equation}
		\label{equ3.9}
		\begin{aligned}
			\left\|F_2(\psi)\right\|_{L^{\frac{6}{5}}\left(\mathbb{R}^3\right)}
			&\leq C(\nu)\left(\|\nabla \psi\|_{L^3\left(\mathbb{R}^3\right)}\|\nabla B\|_{L^2\left(\mathbb{R}^3\right)}+\|\Delta \psi\|_{L^{\frac{3}{2}}\left(\mathbb{R}^3\right)}\|B\|_{L^6\left(\mathbb{R}^3\right)}\right. \\
			&\quad +\left\|\partial_{x_1} \psi\right\|_{L^{\frac{3}{2}}\left(\mathbb{R}^3\right)}\|B\|_{L^6\left(\mathbb{R}^3\right)}+\|B\|_{L^6\left(\mathbb{R}^3\right)}\|\nabla \psi\|_{L^6\left(\mathbb{R}^3\right)}\|\nabla B\|_{L^2\left(\mathbb{R}^3\right)} \\
			&\quad +\left. \|B\|^2_{L^6\left(\mathbb{R}^3\right)}\|\Delta \psi\|_{L^2\left(\mathbb{R}^3\right)}+ \|\psi\|_{L^{\infty}(\mathbb{R}^3)}\|\nabla B\|^{\frac{1}{3}}_{L^{\infty}(\mathbb{R}^3)}\|\nabla B\|^{\frac{5}{3}}_{L^{2}(\mathbb{R}^3)}\right)  \\
			&\leq C(\nu, M)<\infty,
		\end{aligned}
	\end{equation}
	where we  used 
	$$
	\|B\cdot\nabla(\nabla\psi\times B)\|_{L^{\frac{6}{5}}\left(\mathbb{R}^3\right)} \leq \left(\int_{\mathbb{R}^3}|B|^{\frac{12}{5}}|\Delta\psi|^{\frac{6}{5}}dx\right)^\frac{5}{6} + \left(\int_{\mathbb{R}^3}|B|^{\frac{6}{5}}|\nabla\psi|^{\frac{6}{5}}|\nabla B|^{\frac{6}{5}} dx\right)^\frac{5}{6}.
	$$
	Moreover,
	$$
	\begin{aligned}
		\|\nabla \psi \cdot v\|_{W^{1, \frac{6}{5}}\left(\mathbb{R}^3\right)} \leq & \|\nabla \psi\|_{L^{\frac{3}{2}}\left(\mathbb{R}^3\right)}\|v\|_{L^6\left(\mathbb{R}^3\right)}+\left\|\nabla^2 \psi\right\|_{L^{\frac{3}{2}}\left(\mathbb{R}^3\right)}\|v\|_{L^6\left(\mathbb{R}^3)\right.} \\
		& +\|\nabla \psi\|_{L^3\left(\mathbb{R}^3\right)}\|\nabla v\|_{L^2\left(\mathbb{R}^3\right)} \\
		\leq & C(M)<\infty .
	\end{aligned}
	$$
	Let $\varepsilon_0$ as in Lemma \ref{lem2.7}, there exists an $M>0$ such that
	$$
	2\|\nabla v\|_{L^2 (B_M^c)} +2\|\nabla B\|_{L^2 (B_M^c)} +\|B\|_{L^{\infty}(B_M^c)}<\varepsilon_0.
	$$
	Then applying Lemma \ref{lem2.7} to the system (\ref{equ3.8}), choosing $q=\frac{6}{5} $, one derives $\psi v, \psi B \in L^3\left(\mathbb{R}^3\right)$. It follows that
	$$
	\|v\|_{L^3\left(\mathbb{R}^3\right)}+\|B\|_{L^3\left(\mathbb{R}^3\right)} + \|p\|_{L^{\frac{3}{2}}\left(\mathbb{R}^3\right)}<\infty .
	$$

	{\bf Step II. Vanishing of $v$ and $B$.} 
	Multiplying $(\ref{equ3.7})_1$ by $v \varphi_R$,  $(\ref{equ3.7})_2$ by $B \varphi_R$, and integrating over $\mathbb{R}^3$ imply
	$$
	\begin{aligned}
		& \int_{\mathbb{R}^3} \varphi_R\left(\kappa|\nabla v|^2+\nu|\nabla B|^2\right) d x \\
		& =\frac{1}{2} \int_{\mathbb{R}^3} \Delta \varphi_R\left(\kappa|v|^2+\nu|B|^2\right) d x+\frac{1}{2} \int_{\mathbb{R}^3}\left(|v|^2+|B|^2\right) \partial_{x_1} \varphi_R d x \\
		& +\int_{\mathbb{R}^3} p v \cdot \nabla \varphi_R d x+\frac{1}{2} \int_{\mathbb{R}^3}\left[\left(|v|^2+|B|^2\right) v-2(B \cdot v) B\right] \cdot \nabla \varphi_R d x \\
		& -\int_{\mathbb{R}^3} (\nabla\times B)\times B\cdot(B\times \nabla \varphi_R)dx \\
		&:= \sum_{i=1}^5 I_i,
	\end{aligned}
	$$
	where these items are estimated as follows:
	$$
	\begin{aligned}
		\left|I_1\right| 
		%			& \leq C\left\|\Delta \varphi_R\right\|_{L^{\infty}\left(\mathbb{R}^3\right)}\left(\int_{B_{2 R}\setminus B_R}\left(|v|^3+|h|^3\right) d x\right)^{2 / 3}\left(\int_{B_{2 R}\setminus B_R} d x\right)^{1 / 3} \\
		& 
		\lesssim R^{-1}\left(\|v\|_{L^3\left(B_{2 R}\setminus B_R\right)}^2+\|B\|_{L^3\left(B_{2 R}\setminus B_R\right)}^2\right) , \\
		%		\end{aligned} 
	%		$$
	%		$$
	%		\begin{aligned}
		\left|I_2\right| 
		%			& \leq C\left\|\nabla \varphi_R\right\|_{L^{\infty}\left(\mathbb{R}^3\right)}\left(\int_{B_{2 R}\setminus B_R}\left(|v|^3+|h|^3\right) d x\right)^{2 / 3}\left(\int_{B_{2 R}\setminus B_R} d x\right)^{1 / 3} \\
		& \lesssim\|v\|_{L^3\left(B_{2 R}\setminus B_R\right)}^2+\|B\|_{L^3\left(B_{2 R}\setminus B_R\right)}^2 ,\\
		%		\end{aligned} 
	%		$$
	%		$$
	%		\begin{aligned}
		\left|I_3\right| 
		%			& \leq C\left\|\nabla \varphi_R\right\|_{L^{\infty}\left(\mathbb{R}^3\right)}\left(\int_{B_{2 R}\setminus B_R}|p|^{\frac{3}{2}} d x\right)^{2 / 3}\left(\int_{B_{2 R}\setminus B_R}|v|^3 d x\right)^{1 / 3} \\
		& \lesssim R^{-1}\|p\|_{L^{\frac{3}{2}}\left(B_{2 R}\setminus B_R\right)}\|v\|_{L^3\left(B_{2 R}\setminus B_R\right)} , \\
		%		\end{aligned} 
	%		$$
	%		$$
	%		\begin{aligned}
		\left|I_4\right| 
		%			& \leq C\left\|\nabla \varphi_R\right\|_{L^{\infty}\left(\mathbb{R}^3\right)} \int_{B_{2 R}\setminus B_R}\left(|v|^3+|h|^2|v|\right) d x \\
		& \lesssim R^{-1}\left(\|v\|_{L^3\left(B_{2 R}\setminus B_R\right)}^3+\|B\|_{L^3\left(B_{2 R}\setminus B_R\right)}^3\right), \\
		%		\end{aligned} 
	%		$$
	%		$$
	%		\begin{aligned}
		\left|I_5\right| 
		%			& \leq C\left\|\nabla \varphi_R\right\|_{L^{\infty}\left(\mathbb{R}^3\right)} \int_{B_{2 R}\setminus B_R}|\nabla\times B||h|^2 d x \\
		& \lesssim R^{-\frac{1}{2}}\|\nabla B\|_{L^2\left(B_{2 R}\setminus B_R\right)} \|B\|^2_{L^6\left(B_{2 R}\setminus B_R\right)} ,
	\end{aligned}
	$$
	which imply $v=B \equiv 0$ for all $x \in \mathbb{R}^3$ by letting $R \rightarrow \infty$. The proof is complete.
\end{proof}
%	\begin{remark}
	%		Let us explain here that why the method in the proof of Theorem \ref{the1.1} fails when trying to solve Theorem \ref*{the1.2} without the condition $\kappa=\nu$, which seems to be not physical. Let us compare (\ref{equ3.5}) to (\ref{equ3.1}), the only difference happens on the "$\partial_{x_1}-$ terms". However, this only difference results in the linear part of the equations of $u$ and $\eta$ totally mixed each other, which ends up with the failure. But inspired by \cite{wei2017global}, we denote $\mu_1 = \frac{1}{2}(\kappa + \nu) , \mu_2 = \frac{1}{2}(\kappa - \nu)$ and use the Elsässer variables for $u$ and $\eta$ and then we can avoid this problem.
	%	\end{remark}

\section{Degenerate MHD system: Proof of Theorem \ref{the1.6}}
\label{sec5}

	Consider 
	the 3D degenerate Oseen system  with some special critical terms  in   $\mathbb{R}^3$ as follows:
	\begin{equation}
		\label{eq:degenerate oseen}
		\left\{\begin{array}{l}
			\partial_{x_1} w_1+\nabla p-(\frac{1}{2}\Delta_\kappa + \frac{1}{2}\Delta_\nu) w_1 - (\frac{1}{2}\Delta_\kappa - \frac{1}{2}\Delta_\nu) w_2  =f_1, \\
			-\partial_{x_1} w_2+\nabla p-(\frac{1}{2}\Delta_\kappa + \frac{1}{2}\Delta_\nu) w_2 - (\frac{1}{2}\Delta_\kappa - \frac{1}{2}\Delta_\nu) w_1 =f_2,  \\
			\nabla \cdot w_1=g,
		\end{array}\right.
	\end{equation}
	where $\Delta_\kappa$ and $\Delta_\nu$  as defined in (\ref{eq:laplace k}).  Applying Theorem \ref{thm:lizorkin}, we derive a type of $L^q$ estimate for the above degenerate Oseen system of (\ref{eq:degenerate oseen}).
	\begin{Lem}[$L^q$ estimate for degenerate Oseen system]\label{lem5.1}
		Let $f_1, f_2 \in L^q\left(\mathbb{R}^3\right), g \in W^{1, q}\left(\mathbb{R}^3\right)$ with $1<q<2$ and $r=\left(\frac{1}{q}-\frac{1}{2}\right)^{-1}$. Then we have the following results: \\
		(i) when $\kappa_1 = \nu_1=0$, there exists a unique solution $\left(w_1, w_2,p\right)\in \left(L^r\left(\mathbb{R}^3\right)\right)^6 \times\left(\dot{W}^{1, q}\left(\mathbb{R}^3\right) / \mathcal{P}_0\left(\mathbb{R}^3\right)\right)$ of the system (\ref{eq:degenerate oseen}) such that
		\ben
		\label{eq:w1w2 estimate'}
		&& \left\|\left(w_1, w_2\right)\right\|_{L^r\left(\mathbb{R}^3\right)}+\|\nabla p\|_{L^q\left(\mathbb{R}^3\right)} \nonumber\\
		& \leq &C(q, \kappa_2, \kappa_3, \nu_2, \nu_3)\left(\left\|f_1\right\|_{L^q\left(\mathbb{R}^3\right)}+\left\|f_2\right\|_{L^q\left(\mathbb{R}^3\right)}+\|g\|_{W^{1, q}\left(\mathbb{R}^3\right)}\right) .
		\een
		(ii) when $\kappa_1>0$ and $\nu_1>0$, there exists a unique solution $$\left(w_1, w_2, p\right)\in\left(\dot{W}^{2, q}\left(\mathbb{R}^3\right) \cap L^r\left(\mathbb{R}^3\right)\right)^6 \times\left(\dot{W}^{1, q}\left(\mathbb{R}^3\right) / \mathcal{P}_0\left(\mathbb{R}^3\right)\right)$$ of the system (\ref{eq:degenerate oseen}) such that
		\ben
		&& \left\|\left(\nabla^2 w_1, \nabla^2 w_2\right)\right\|_{L^q\left(\mathbb{R}^3\right)}+\left\|\left(w_1, w_2\right)\right\|_{L^r\left(\mathbb{R}^3\right)}+\|\nabla p\|_{L^q\left(\mathbb{R}^3\right)} \nonumber\\
		& \leq &C(q, \kappa_1, \kappa_2, \kappa_3, \nu_1, \nu_2, \nu_3)\left(\left\|f_1\right\|_{L^q\left(\mathbb{R}^3\right)}+\left\|f_2\right\|_{L^q\left(\mathbb{R}^3\right)}+\|g\|_{W^{1, q}\left(\mathbb{R}^3\right)}\right) .
		\een
	\end{Lem}	
	\begin{proof}
		Denoting $\hat{\kappa}=(\sqrt{\kappa_1}, \sqrt{\kappa_2}, \sqrt{\kappa_3})$ and $\hat{\nu} = (\sqrt{\nu_1}, \sqrt{\nu_2}, \sqrt{\nu_3})$, performing a Fourier transform on the system of \eqref{eq:degenerate oseen}, we obtain that
		\begin{equation}
			\label{eq:w1w2'}
			\left\{\begin{array}{l}
				(\frac12 |\hat{\kappa}\cdot\xi|^2 +\frac12|\hat{\nu}\cdot\xi|^2 + i\xi_1) \hat{w_1} + (\frac12 |\hat{\kappa}\cdot\xi|^2 - \frac12|\hat{\nu}\cdot\xi|^2)\hat{w_2} + i\xi \hat{p} = \hat{f_1},\\
				(\frac12 |\hat{\kappa}\cdot\xi|^2 +\frac12|\hat{\nu}\cdot\xi|^2 - i\xi_1) \hat{w_2} + (\frac12 |\hat{\kappa}\cdot\xi|^2 -\frac12|\hat{\nu}\cdot\xi|^2)\hat{w_1} +   i\xi \hat{p} = \hat{f_2},\\
				i \xi \cdot \hat{w_1}=\hat{g}.
			\end{array}\right.  
		\end{equation}
		Taking $i\xi$ on the both sides of  $(\ref{eq:w1w2'})_1$ and $(\ref{eq:w1w2'})_2$ as the vector inner product, we get
		\begin{equation}
			\label{equ2.5'}
			\left\{\begin{array}{l}
				(\frac12 |\hat{\kappa}\cdot\xi|^2 +\frac12|\hat{\nu}\cdot\xi|^2 + i\xi_1) \hat{g} + (\frac12 |\hat{\kappa}\cdot\xi|^2 -\frac12|\hat{\nu}\cdot\xi|^2) i\xi\cdot\hat{w_2}-|\xi|^2\hat{p} = i\xi\cdot\hat{f_1},\\
				(\frac12 |\hat{\kappa}\cdot\xi|^2 +\frac12|\hat{\nu}\cdot\xi|^2 - i\xi_1) i\xi\cdot \hat{w_2} + (\frac12 |\hat{\kappa}\cdot\xi|^2 -\frac12|\hat{\nu}\cdot\xi|^2 )\hat{g} -|\xi|^2 \hat{p} = i\xi\cdot\hat{f_2},
			\end{array}\right. 
		\end{equation}
		which implies
		\begin{equation}
			\begin{aligned}
				\label{eq:pressure'}
				\hat{p} = \frac{1}{|\hat{\nu}\cdot\xi|^2- i\xi_1}\bigg\{\left(|\hat{\kappa}\cdot\xi|^2|\hat{\nu}\cdot\xi|^2+\xi_1^2\right)\frac{\hat{g}}{|\xi|^2} - (\frac12 |\hat{\kappa}\cdot\xi|^2 +\frac12|\hat{\nu}\cdot\xi|^2 - i\xi_1)\frac{i\xi\cdot\hat{f_1}}{|\xi|^2} \\
				+(\frac12 |\hat{\kappa}\cdot\xi|^2 -\frac12|\hat{\nu}\cdot\xi|^2 )\frac{i\xi\cdot\hat{f_2}}{|\xi|^2}\bigg\}.
			\end{aligned}
		\end{equation}
		By substituting (\ref{eq:pressure'}) into $(\ref{eq:w1w2'})$, we get
		\ben
		\label{eq:G'}
		\hat{w_1} &=& \frac{\frac12 |\hat{\kappa}\cdot\xi|^2 +\frac12|\hat{\nu}\cdot\xi|^2 - i\xi_1}{|\hat{\kappa}\cdot\xi|^2|\hat{\nu}\cdot\xi|^2+\xi_1^2}\left(id - \frac{\xi\otimes\xi}{|\xi|^2}\right)\hat{f_1} \\ \nonumber
		&&+ \frac{\frac12 |\hat{\kappa}\cdot\xi|^2 -\frac12|\hat{\nu}\cdot\xi|^2 }{|\hat{\kappa}\cdot\xi|^2|\hat{\nu}\cdot\xi|^2+\xi_1^2}\left( \frac{\xi\otimes\xi}{|\xi|^2} - id\right)\hat{f_2} -\frac{i\xi\hat{g}}{|\xi|^2},	
		\een
		\ben
		\label{eq:G''}
		\hat{w_2} &=& \frac{|\hat{\nu}\cdot\xi|^2 + i\xi_1}{|\hat{\nu}\cdot\xi|^2 - i\xi_1}\hat{w_1}-\frac{\hat{f_1} -\hat{f_2} }{|\hat{\nu}\cdot\xi|^2 - i\xi_1}\\ \nonumber&=& \left(-\frac{\frac12 |\hat{\kappa}\cdot\xi|^2 -\frac12|\hat{\nu}\cdot\xi|^2 }{|\hat{\kappa}\cdot\xi|^2|\hat{\nu}\cdot\xi|^2+\xi_1^2}\right)\hat{f_1} \\ \nonumber
		&&-\left( \frac{|\hat{\nu}\cdot\xi|^2 + i\xi_1}{|\hat{\nu}\cdot\xi|^2 - i\xi_1}\frac{\frac12 |\hat{\kappa}\cdot\xi|^2 +\frac12|\hat{\nu}\cdot\xi|^2 - i\xi_1}{|\hat{\kappa}\cdot\xi|^2|\hat{\nu}\cdot\xi|^2+\xi_1^2}\frac{\xi\otimes\xi\cdot}{|\xi|^2}\right)\hat{f_1}\\ \nonumber
		&&+ \left(\frac{\frac12 |\hat{\kappa}\cdot\xi|^2 +\frac12|\hat{\nu}\cdot\xi|^2+i\xi }{|\hat{\kappa}\cdot\xi|^2|\hat{\nu}\cdot\xi|^2+\xi_1^2}\right)\hat{f_2} \\ \nonumber
		&&+ \left(\frac{|\hat{\nu}\cdot\xi|^2 + i\xi_1}{|\hat{\nu}\cdot\xi|^2 - i\xi_1}\frac{\frac12 |\hat{\kappa}\cdot\xi|^2 -\frac12|\hat{\nu}\cdot\xi|^2}{|\hat{\kappa}\cdot\xi|^2|\hat{\nu}\cdot\xi|^2+\xi_1^2}\frac{\xi\otimes\xi\cdot}{|\xi|^2}\right)\hat{f_2}\\ \nonumber	
		&&-\frac{|\hat{\nu}\cdot\xi|^2 + i\xi_1}{|\hat{\nu}\cdot\xi|^2 - i\xi_1}\frac{i\xi\hat{g}}{|\xi|^2}.
		\een
		(i)	For $\kappa_1=\nu_1=0$, first we consider 
		$$
		\Phi_{\kappa, \nu}^{1}(\xi):=\frac{\kappa_2\xi_2^2+\kappa_3\xi_3^2}{(\kappa_2\xi_2^2+\kappa_3\xi_3^2)(\nu_2\xi_2^2+\nu_3\xi_3^2)+\xi_1^2},
		$$
		$$
		\Phi_{\kappa, \nu}^2(\xi):=\frac{\nu_2\xi_2^2+\nu_3\xi_3^2}{(\kappa_2\xi_2^2+\kappa_3\xi_3^2)(\nu_2\xi_2^2+\nu_3\xi_3^2)+\xi_1^2}
		$$
		and
		$$
		\Phi_{\kappa, \nu}^3(\xi):=\frac{i\xi_1}{(\kappa_2\xi_2^2+\kappa_3\xi_3^2)(\nu_2\xi_2^2+\nu_3\xi_3^2)+\xi_1^2},
		$$
		which satisfy the conditions in Theorem \ref{thm:lizorkin} with $\beta=\frac{1}{2}$. In fact, 
		\beno
		&&\left|\xi_1\right|^{1 / 2}\left|\xi_2\right|^{1 / 2}\left|\xi_3\right|^{1 / 2}|\Phi_{\kappa, \nu}^{1}(\xi)|\\
		&\leq& C (\kappa_2, \kappa_3)\frac{\xi_1^2+\xi_2^4+\xi_3^4 }{(\kappa_2\xi_2^2+\kappa_3\xi_3^2)(\nu_2\xi_2^2+\nu_3\xi_3^2)+\xi_1^2} \leq C (\kappa_2, \kappa_3,\nu_2,\nu_3),
		\eeno
		and $\Phi_{\kappa, \nu}^2(\xi)$, $\Phi_{\kappa, \nu}^3(\xi)$ are similar. Second, for the term of  $-\frac{i\xi\hat{g}(\xi)}{|\xi|^2}$, it is similar as in  \eqref{eq:decomposition}. Hence, by Theorem \ref{thm:lizorkin}, we get
		\ben\label{eq:v,h'}
		\|(w_1,w_2)\|_{L^r\left(\mathbb{R}^3\right)} \leq C(q, \kappa_2, \kappa_3,\nu_2,\nu_3)\left(\left\|f_1\right\|_{L^q\left(\mathbb{R}^3\right)}+\left\|f_2\right\|_{L^q\left(\mathbb{R}^3\right)}+\|g\|_{W^{1, q}\left(\mathbb{R}^3\right)}\right) .\een
		Finally, for the pressure $i\xi\hat{p}$,  we get
		\beno
		\begin{aligned}
			i\xi\hat{p} = \frac{1}{|\hat{\nu}\cdot\xi|^2- i\xi_1}\bigg\{\left(|\hat{\kappa}\cdot\xi|^2|\hat{\nu}\cdot\xi|^2+\xi_1^2\right)\frac{i\xi\hat{g}}{|\xi|^2} + (\frac12 |\hat{\kappa}\cdot\xi|^2 +\frac12|\hat{\nu}\cdot\xi|^2 - i\xi_1)\frac{\xi\otimes\xi\cdot\hat{f_1}}{|\xi|^2} \\
			-(\frac12 |\hat{\kappa}\cdot\xi|^2 -\frac12|\hat{\nu}\cdot\xi|^2 )\frac{\xi\otimes\xi\cdot\hat{f_2}}{|\xi|^2}\bigg\},
		\end{aligned}
		\eeno
		which implies
		\ben\label{eq:p only'}
		\|\nabla p\|_{L^q\left(\mathbb{R}^3\right)} \leq C(q, \kappa_2, \kappa_3, \nu_2, \nu_3)\left(\left\|f_1\right\|_{L^q\left(\mathbb{R}^3\right)}+\left\|f_2\right\|_{L^q\left(\mathbb{R}^3\right)}+\|g\|_{W^{1, q}\left(\mathbb{R}^3\right)}\right) .\een
		(ii) For $
		\kappa_1, \nu_1>0$, first we consider 
		$$
		\Phi_{\kappa, \nu}^{4}(\xi):=\frac{\kappa_1\xi_1^2+\kappa_2\xi_2^2+\kappa_3\xi_3^2}{(\kappa_1\xi_1^2+\kappa_2\xi_2^2+\kappa_3\xi_3^2)(\nu_1\xi_1^2+\nu_2\xi_2^2+\nu_3\xi_3^2)+\xi_1^2},
		$$
		$$
		\Phi_{\kappa, \nu}^5(\xi):=\frac{\nu_1\xi_1^2+\nu_2\xi_2^2+\nu_3\xi_3^2}{(\kappa_1\xi_1^2+\kappa_2\xi_2^2+\kappa_3\xi_3^2)(\nu_1\xi_1^2+\nu_2\xi_2^2+\nu_3\xi_3^2)+\xi_1^2}
		$$
		and
		$$
		\Phi_{\kappa, \nu}^6(\xi):=\frac{i\xi_1}{(\kappa_1\xi_1^2+\kappa_2\xi_2^2+\kappa_3\xi_3^2)(\nu_1\xi_1^2+\nu_2\xi_2^2+\nu_3\xi_3^2)+\xi_1^2}
		$$
		which satisfy the conditions in Theorem \ref{thm:lizorkin} with $\beta=\frac{1}{2}$. Second, for the term of  $-\frac{i\xi\hat{g}(\xi)}{|\xi|^2}$, it is similar as in  \eqref{eq:decomposition}. Hence, we get
		\beno
		\|(w_1,w_2)\|_{L^r\left(\mathbb{R}^3\right)} \leq C(q, \kappa_1, \kappa_2, \kappa_3, \nu_1, \nu_2, \nu_3)\left(\left\|f_1\right\|_{L^q\left(\mathbb{R}^3\right)}+\left\|f_2\right\|_{L^q\left(\mathbb{R}^3\right)}+\|g\|_{W^{1, q}\left(\mathbb{R}^3\right)}\right) .\eeno
		Finally, for the pressure $i\xi\hat{p}$,  we get
		\beno
		\|\nabla p\|_{L^q\left(\mathbb{R}^3\right)} \leq C(q, \kappa_1, \kappa_2, \kappa_3, \nu_1, \nu_2, \nu_3)\left(\left\|f_1\right\|_{L^q\left(\mathbb{R}^3\right)}+\left\|f_2\right\|_{L^q\left(\mathbb{R}^3\right)}+\|g\|_{W^{1, q}\left(\mathbb{R}^3\right)}\right) .\eeno
		Moreover, $\xi_i\xi_j\Phi_{\kappa, \nu}^{4}(\xi)$, $\xi_i\xi_j\Phi_{\kappa, \nu}^{5}(\xi)$ and $\xi_i\xi_j\Phi_{\kappa, \nu}^{6}(\xi)$ satisfy the conditions of Theorem \ref{thm:lizorkin} with $\beta=0$. The proof is complete. 
	\end{proof}
	
	\begin{Lem}
		\label{lem5.2}
		Assume  the same conditions of $q, f_1, f_2, g, \mathcal{M}$ as in Lemma \ref{lem2.6}. In addition, let $\kappa_1, \nu_1\geq 0$, $\kappa_2,\kappa_3,\nu_2,\nu_3>0$. Let $(v, B, p)$ satisfy the following perturbation degenerate Oseen system:
		\begin{equation}
			\label{equ2.18}
			\left\{\begin{aligned}
				-\left(\begin{array}{l}
					\Delta_\kappa v \\
					\Delta_\nu B
				\end{array}\right)+\partial_{x_1}\left(\begin{array}{l}
					v \\
					B
				\end{array}\right)+\mathcal{M}\left(\begin{array}{l}
					v \\
					B
				\end{array}\right)+\left(\begin{array}{c}
					\nabla p \\
					0
				\end{array}\right)&=\left(\begin{array}{l}
					f_1 \\
					f_2
				\end{array}\right),\\
				\nabla \cdot v&=g.
			\end{aligned}\right. 
		\end{equation}
	    Then, we have the following results:\\
		(i) when $\kappa_1\nu_1= 0$, there exists a small constant $\varepsilon_1 >0$ such that if
		$$
		\|\mathcal{M}\|_{L^2\left(\mathbb{R}^3\right)}<\varepsilon_1,
		$$
		there exists a unique solution $(v, B, p) \in\left(L^r\left(\mathbb{R}^3\right)\right)^6 \times\left(\dot{W}^{1, q}\left(\mathbb{R}^3\right) / \mathcal{P}_0\left(\mathbb{R}^3\right)\right)$ of \eqref{equ2.18} such that
		$$
		\begin{aligned}
			&\quad\|(v, B)\|_{L^r\left(\mathbb{R}^3\right)}+\|\nabla p\|_{L^q\left(\mathbb{R}^3\right)} \\
			& \quad \leq C(q, \varepsilon_1, \kappa_1, \kappa_2, \kappa_3, \nu_2, \nu_3)\left(\left\|f_1\right\|_{L^q\left(\mathbb{R}^3\right)}+\left\|f_2\right\|_{L^q\left(\mathbb{R}^3\right)}+\|g\|_{W^{1, q}\left(\mathbb{R}^3\right)}\right),
		\end{aligned}
		$$
		(ii) when $\kappa_1\nu_1\neq 0$, there exists a small constant $\varepsilon_2 >0$ such that if
		$$
		\|\mathcal{M}\|_{L^2\left(\mathbb{R}^3\right)}<\varepsilon_2,
		$$ there exists a unique solution $$(v, B, p) \in\left(\dot{W}^{2, q}\left(\mathbb{R}^3\right) \cap L^r\left(\mathbb{R}^3\right)\right)^6 \times\left(\dot{W}^{1, q}\left(\mathbb{R}^3\right) / \mathcal{P}_0\left(\mathbb{R}^3\right)\right)$$  of \eqref{equ2.18} such that
		$$
		\begin{aligned}
			& \left\|\left(\nabla^2 v, \nabla^2 B\right)\right\|_{L^q\left(\mathbb{R}^3\right)}+\|(v, B)\|_{L^r\left(\mathbb{R}^3\right)}+\|\nabla p\|_{L^q\left(\mathbb{R}^3\right)} \\
			& \quad \leq C(q,\varepsilon_2, \kappa_1, \kappa_2, \kappa_3, \nu_1, \nu_2, \nu_3)\left(\left\|f_1\right\|_{L^q\left(\mathbb{R}^3\right)}+\left\|f_2\right\|_{L^q\left(\mathbb{R}^3\right)}+\|g\|_{W^{1, q}\left(\mathbb{R}^3\right)}\right),
		\end{aligned}
		$$
		where $r=\left(\frac{1}{q}-\frac{1}{2}\right)^{-1}$.
	\end{Lem}
	\begin{proof}
		Thanks to Lemma \ref{lem2.3}, the proof of Lemma \ref{lem5.2} is similar to the proof of Lemma \ref{lem2.6} and we omit it. 
	\end{proof}
	
	\begin{Lem}
		\label{lem5.3}
		Assume the same conditions of $q, f_1, f_2, g, \mathcal{M}$ in Lemma \ref{lem2.6}. In addition, $\kappa_1, \nu_1\geq 0$, $\kappa_2,\kappa_3,\nu_2,\nu_3>0$. $(w_1, w_2, p)$ satisfies the following perturbation degenerate Oseen system:
		\begin{equation}
			\label{equ2.18'}
			\left\{\begin{aligned}
				- (\frac{1}{2}\Delta_\kappa + \frac{1}{2}\Delta_\nu)\left(\begin{array}{l}
					w_1 \\
					w_2
				\end{array}\right)- (\frac{1}{2}\Delta_\kappa - \frac{1}{2}\Delta_\nu)\left(\begin{array}{l}
					w_1 \\
					w_2
				\end{array}\right)\\+\partial_{x_1}\left(\begin{array}{l}
					w_1 \\
					-w_2
				\end{array}\right)
				+\mathcal{M}\left(\begin{array}{l}
					w_1 \\
					w_2
				\end{array}\right)+\left(\begin{array}{c}
					\nabla p \\
					\nabla p
				\end{array}\right)&=\left(\begin{array}{l}
					f_1 \\
					f_2
				\end{array}\right),\\
				\nabla \cdot v&=g.
			\end{aligned}\right. 
		\end{equation}
	    Then, we have the following results:\\
		(i) when $\kappa_1=\nu_1= 0$, there exists a small constant $\varepsilon_1 >0$ such that if
		$$
		\|\mathcal{M}\|_{L^2\left(\mathbb{R}^3\right)}<\varepsilon_1,
		$$
		there exists a unique solution $(w_1, w_2, p) \in\left(L^r\left(\mathbb{R}^3\right)\right)^6 \times\left(\dot{W}^{1, q}\left(\mathbb{R}^3\right) / \mathcal{P}_0\left(\mathbb{R}^3\right)\right)$ of \eqref{equ2.18'} such that
		$$
		\begin{aligned}
			&\quad\|(w_1, w_2)\|_{L^r\left(\mathbb{R}^3\right)}+\|\nabla p\|_{L^q\left(\mathbb{R}^3\right)} \\
			& \quad \leq C(q, \varepsilon_1, \kappa_2, \kappa_3, \nu_2, \nu_3)\left(\left\|f_1\right\|_{L^q\left(\mathbb{R}^3\right)}+\left\|f_2\right\|_{L^q\left(\mathbb{R}^3\right)}+\|g\|_{W^{1, q}\left(\mathbb{R}^3\right)}\right),
		\end{aligned}
		$$
		(ii) when $\kappa_1\nu_1\neq 0$, there exists a small constant $\varepsilon_2 >0$ such that if
		$$
		\|\mathcal{M}\|_{L^2\left(\mathbb{R}^3\right)}<\varepsilon_2,
		$$ there exists a unique solution $$(w_1, w_2, p) \in\left(\dot{W}^{2, q}\left(\mathbb{R}^3\right) \cap L^r\left(\mathbb{R}^3\right)\right)^6 \times\left(\dot{W}^{1, q}\left(\mathbb{R}^3\right) / \mathcal{P}_0\left(\mathbb{R}^3\right)\right)$$ of \eqref{equ2.18'} such that
		$$
		\begin{aligned}
			& \left\|\left(\nabla^2 w_1, \nabla^2 w_2\right)\right\|_{L^q\left(\mathbb{R}^3\right)}+\|(w_1, w_2)\|_{L^r\left(\mathbb{R}^3\right)}+\|\nabla p\|_{L^q\left(\mathbb{R}^3\right)} \\
			& \quad \leq C(q,\varepsilon_2, \kappa_1, \kappa_2, \kappa_3, \nu_1, \nu_2, \nu_3)\left(\left\|f_1\right\|_{L^q\left(\mathbb{R}^3\right)}+\left\|f_2\right\|_{L^q\left(\mathbb{R}^3\right)}+\|g\|_{W^{1, q}\left(\mathbb{R}^3\right)}\right),
		\end{aligned}
		$$
		where $r=\left(\frac{1}{q}-\frac{1}{2}\right)^{-1}$.
	\end{Lem}
	\begin{proof}
		Thanks to Lemma \ref{lem5.1}, the proof of Lemma \ref{lem5.3} is similar to the proof of Lemma \ref{lem2.6} and we omit it. 
	\end{proof}
	
	\begin{proof}[Proof of Theorem \ref{the1.6}.]
		{\bf \underline{Case I. $u_{\infty}=(1,0,0), B_{\infty} = 0 $}.}  
		%We may assume $u_{\infty}=(1,0,0)^T$ without loss of generality. 
		Denoting $v:=u-u_{\infty}$, it follows from \eqref{equ1.3} that
		\begin{equation}
			\label{equ3.10}
			\left\{\begin{array}{l}
				v \cdot \nabla v+\partial_{x_1} v+\nabla p-B \cdot \nabla B- \Delta_\kappa v=0, \\
				v \cdot \nabla B+\partial_{x_1} B-B \cdot \nabla v- \Delta_\nu B=0, \\
				\nabla \cdot v=0, \quad \nabla \cdot B=0.
			\end{array}\right.
		\end{equation}
		Multiplying $\psi$, defined  in (\ref{eq:psi}), on the both sides of (\ref{equ3.10}), it follows that
		\ben
		\left\{\begin{aligned}
			\partial_{x_1}\left(\begin{array}{c}
				\psi v \\
				\psi B
			\end{array}\right)+\left(\begin{array}{c}
				(\nabla v) \chi_{B_M^c},-(\nabla B) \chi_{B_M^c} \\
				(\nabla B) \chi_{B_M^c},-(\nabla v) \chi_{B_M^c}
			\end{array}\right) \cdot\left(\begin{array}{c}
				\psi v \\
				\psi B
			\end{array}\right)\\-\left(\begin{array}{c}
				\Delta_\kappa \psi v \\
				\Delta_\nu \psi B
			\end{array}\right)+\left(\begin{array}{c}
				\nabla(\psi p) \\
				0
			\end{array}\right) 
			&=\left(\begin{array}{c}
				F_1^{\prime}(\psi) \\
				F_2^{\prime}(\psi)
			\end{array}\right), \\
			\nabla \cdot(\psi v)&=\nabla \psi \cdot v,
		\end{aligned}\right. 
		\een
		where
		$$
		\begin{aligned}
			& F_1^{\prime}(\psi)=-2  \nabla_\kappa \psi \cdot \nabla_\kappa v-(\Delta_\kappa \psi) v+v \partial_{x_1} \psi+(\nabla \psi) p, \\
			& F_2^{\prime}(\psi)=-2  \nabla_\nu \psi \cdot \nabla_\nu B-(\Delta_\nu \psi) B+B\partial_{x_1} \psi
		\end{aligned}
		$$
		and $\nabla_\kappa := (\sqrt{\kappa_1}\partial_{x_1}, \sqrt{\kappa_2}\partial_{x_2}, \sqrt{\kappa_3}\partial_{x_3})$, $\nabla_\nu := (\sqrt{\nu_1}\partial_{x_1}, \sqrt{\nu_2}\partial_{x_2}, \sqrt{\nu_3}\partial_{x_3})$.
		In the same way as  the proof of Theorem \ref{the1.1}, one can find 
		$$
		\begin{aligned}
			\left\|F_1^{\prime}(\psi)\right\|_{L^{\frac{6}{5}}\left(\mathbb{R}^3\right)} +\left\|F_2^{\prime}(\psi)\right\|_{L^{\frac{6}{5}}\left(\mathbb{R}^3\right)}+\|\nabla \psi \cdot v\|_{W^{1, \frac{6}{5}}\left(\mathbb{R}^3\right)}  < \infty.
		\end{aligned}
		$$
		Let $\varepsilon_1, \varepsilon_2$ as in Lemma \ref{lem5.2}, 
		and by the D-condition, there exists an $M>0$ such that
		$$
		2\|\nabla v\|_{L^2 (B_M^c)} +2\|\nabla B\|_{L^2 (B_M^c)} <\text{min}\{\varepsilon_1, \varepsilon_2\}.
		$$
		Then it follows from Lemma \ref{lem5.2} and the smoothness of solutions that
		\ben\label{eq:vh}
		\|v\|_{L^{6}\left(\mathbb{R}^3\right)} + \|B\|_{L^{6}\left(\mathbb{R}^3\right)} + \|v\|_{L^{3}\left(\mathbb{R}^3\right)} + \|B\|_{L^{3}\left(\mathbb{R}^3\right)} + \|p\|_{L^{3}\left(\mathbb{R}^3\right)} +\|p\|_{L^{\frac{3}{2}}\left(\mathbb{R}^3\right)} < \infty.
		\een
		Multiply $(\ref{equ3.10})_1$ by $v \varphi_R$,  $(\ref{equ3.10})_2$ by $B \varphi_R$ and we have
		\beno
		&& \int_{\mathbb{R}^3} \varphi_R\left(|\nabla_\kappa v|^2+|\nabla_\nu B|^2\right) d x \\
		&=& \frac{1}{2} \int_{\mathbb{R}^3} \left(|v|^2 \Delta_\kappa \varphi_R+|B|^2\Delta_\nu \varphi_R\right) d x+\frac{1}{2} \int_{\mathbb{R}^3}\left(|v|^2+|B|^2\right) \partial_{x_1} \varphi_R d x \\
		& &+\int_{\mathbb{R}^3} p v \cdot \nabla \varphi_R d x+\frac{1}{2} \int_{\mathbb{R}^3}\left[\left(|v|^2+|B|^2\right) v-2(B \cdot v) B\right] \cdot \nabla \varphi_R d x \\
		&:=&\sum_{i=1}^4 I_i.
		\eeno
		By H\"{o}lder's inequality,
		$$
		\begin{aligned}
			\left|I_1\right| 
			%    	& \leq C\left\|\Delta \varphi_R\right\|_{L^{\infty}\left(\mathbb{R}^3\right)}\left(\int_{B_{2 R}\SETMINUS B_R}\left(|v|^3+|h|^3\right) d x\right)^{2 / 3}\left(\int_{B_{2 R}\SETMINUS B_R} d x\right)^{1 / 3} \\
			& \lesssim R^{-1}\left(\|v\|_{L^3\left(B_{2 R}\setminus B_R\right)}^2+\|B\|_{L^3\left(B_{2 R}\setminus B_R\right)}^2\right),\\
			\left|I_2\right| 
			%    	& \leq C\left\|\nabla \varphi_R\right\|_{L^{\infty}\left(\mathbb{R}^3\right)}\left(\int_{B_{2 R}\SETMINUS B_R}\left(|v|^3+|h|^3\right) d x\right)^{2 / 3}\left(\int_{B_{2 R}\SETMINUS B_R} d x\right)^{1 / 3} \\
			& \lesssim\|v\|_{L^3\left(B_{2 R}\setminus B_R\right)}^2+\|B\|_{L^3\left(B_{2 R}\setminus B_R\right)}^2, \\
			\left|I_3\right| 
			%    	& \leq C\left\|\nabla \varphi_R\right\|_{L^{\infty}\left(\mathbb{R}^3\right)}\left(\int_{B_{2 R}\SETMINUS B_R}|p|^{\frac{3}{2}} d x\right)^{2 / 3}\left(\int_{B_{2 R}\SETMINUS B_R}|v|^3 d x\right)^{1 / 3} \\
			& \lesssim R^{-1}\|p\|_{L^{\frac{3}{2}}\left(B_{2 R}\setminus B_R\right)}\|v\|_{L^3\left(B_{2 R}\setminus B_R\right)}, \\
			\left|I_4\right| 
			%    	& \leq C\left\|\nabla \varphi_R\right\|_{L^{\infty}\left(\mathbb{R}^3\right)} \int_{B_{2 R}\SETMINUS B_R}\left(|v|^3+|h|^2|v|\right) d x \\
			& \lesssim R^{-1}\left(\|v\|_{L^3\left(B_{2 R}\setminus B_R\right)}^3+\|B\|_{L^3\left(B_{2 R}\setminus B_R\right)}^3\right) .
		\end{aligned}
		$$
		Let $R \rightarrow \infty$, we find
		$$
		\int_{\mathbb{R}^3}\left(|\nabla_\kappa v|^2+|\nabla_\nu B|^2\right) dx=0,
		$$
		which implies the following results:\\
		(i)when $\kappa_1\nu_1\neq 0$, $v\equiv v(x_1)$ and $B\equiv B(x_1)$;\\
		(ii)when $\kappa_1>0$ and $\nu_1=0$, $v\equiv v(x_1)$ and $B\equiv 0$;\\
		(iii)when $\kappa_1=0$ and $\nu_1> 0$, $v\equiv 0$ and $B\equiv B(x_1)$;\\
		(iv)when $\kappa_1=\nu_1\equiv 0$, $v= B\equiv 0$.\\
		Due to \eqref{eq:vh}, for any $\kappa_1\geq 0$, $\nu_1\geq 0$, it immediately follows that $ v\equiv0$ and $B\equiv0$.\\

		{\bf \underline{Case II. $u_{\infty}=0, B_{\infty} =  (-1,0,0)$}.}  Denoting $\eta:=B-B_{\infty}$, it follows from \eqref{equ1.3} that
		\begin{equation}\label{eq5.16}
			\left\{\begin{array}{l}
				u \cdot \nabla u+\partial_{x_1} \eta+\nabla p-\eta \cdot \nabla \eta- \Delta_\kappa u=0, \\
				u \cdot \nabla \eta+\partial_{x_1} u-\eta \cdot \nabla u-  \Delta_\nu \eta=0, \\
				\nabla \cdot u=0, \quad \nabla \cdot \eta=0.
			\end{array}\right.
		\end{equation}
		Hence, taking $w_1=u+\eta$ and $w_2=u-\eta$ we have
		\begin{equation}
			\left\{\begin{array}{l}
				w_2 \cdot \nabla w_1+\partial_{x_1} w_1+\nabla p - (\frac{1}{2}\Delta_\kappa + \frac{1}{2}\Delta_\nu)w_1 - (\frac{1}{2}\Delta_\kappa - \frac{1}{2}\Delta_\nu)w_2 =0, \\
				w_1 \cdot \nabla w_2-\partial_{x_1} w_2+\nabla p- (\frac{1}{2}\Delta_\kappa + \frac{1}{2}\Delta_\nu)w_2 - (\frac{1}{2}\Delta_\kappa - \frac{1}{2}\Delta_\nu)w_1 =0, \\
				\nabla \cdot w_1=0, \quad \nabla \cdot w_2=0, \\
				\int_{\mathbb{R}^3}\left|\nabla w_1\right|^2 d x<+\infty, \quad \int_{\mathbb{R}^3}\left|\nabla w_2\right|^2 d x<+\infty.
			\end{array}\right.
		\end{equation}
		Multiplying $\psi$, defined  in (\ref{eq:psi}), on the both sides of (\ref{eq5.16}), it follows that
		\begin{equation}
			\left\{\begin{aligned}
				\partial_{x_1}\left(\begin{array}{c}
					\psi w_1 \\
					-\psi w_2
				\end{array}\right)-(\frac{1}{2}\Delta_\kappa + \frac{1}{2}\Delta_\nu)\left(\begin{array}{c}
					\psi w_1 \\
					\psi w_2
				\end{array}\right)-(\frac{1}{2}\Delta_\kappa - \frac{1}{2}\Delta_\nu)\left(\begin{array}{c}
					\psi w_2 \\
					\psi w_1
				\end{array}\right)\\
				+\left(\begin{array}{cc}
					0 & \left(\nabla w_1\right)^T \chi_{B_M^c} \\
					\left(\nabla w_2\right)^T \chi_{B_M^c} & 0
				\end{array}\right) \cdot\left(\begin{array}{l}
					\psi w_1 \\
					\psi w_2
				\end{array}\right)
				+\left(\begin{array}{c}
					\nabla(\psi p) \\
					\nabla(\psi p)
				\end{array}\right)
				&=\left(\begin{array}{c}
					G_1'(\psi) \\
					G_2'(\psi)
				\end{array}\right), \\
				\nabla \cdot\left(\psi w_1\right)&=\nabla \psi \cdot w_1,
			\end{aligned}\right.
		\end{equation}
		where
		$$
		\begin{aligned}
			G_1'(\psi)=&- \nabla_\kappa \psi \cdot \nabla_\kappa w_1-\frac12(\Delta_\kappa \psi) w_1- \nabla_\nu \psi \cdot \nabla_\nu w_1-\frac12(\Delta_\nu \psi) w_1\\
			&- \nabla_\kappa \psi \cdot \nabla_\kappa w_2-\frac12(\Delta_\kappa \psi) w_2+ \nabla_\nu \psi \cdot \nabla_\nu w_2+\frac12(\Delta_\nu \psi) w_2\\
			&+w_1 \partial_{x_1} \psi+(\nabla \psi) p, \\
			G_2'(\psi)=&- \nabla_\kappa \psi \cdot \nabla_\kappa w_2-\frac12(\Delta_\kappa \psi) w_2- \nabla_\nu \psi \cdot \nabla_\nu w_2-\frac12(\Delta_\nu \psi) w_2\\
			&- \nabla_\kappa \psi \cdot \nabla_\kappa w_1-\frac12(\Delta_\kappa \psi) w_1+ \nabla_\nu \psi \cdot \nabla_\nu w_1+\frac12(\Delta_\nu \psi) w_1\\
			&-w_2 \partial_{x_1} \psi+(\nabla \psi) p,
		\end{aligned}
		$$
		with
		$$
		\left\|G_1'(\psi)\right\|_{L^{6 / 5}\left(\mathbb{R}^3\right)}+\left\|G_2'(\psi)\right\|_{L^{6 / 5}\left(\mathbb{R}^3\right)}+\left\|\nabla \psi \cdot w_1\right\|_{W^{1,6 / 5}\left(\mathbb{R}^3\right)}< \infty.
		$$
		Then it follows from Lemma \ref{lem5.3} and the smoothness of solutions that
		\beno
		\|u\|_{L^{6}\left(\mathbb{R}^3\right)} + \|\eta\|_{L^{6}\left(\mathbb{R}^3\right)} + \|u\|_{L^{3}\left(\mathbb{R}^3\right)} + \|\eta\|_{L^{3}\left(\mathbb{R}^3\right)} + \|p\|_{L^{3}\left(\mathbb{R}^3\right)} +\|p\|_{L^{\frac{3}{2}}\left(\mathbb{R}^3\right)} < \infty.
		\eeno
		By the same process as Case I, we get
		$$
		\int_{\mathbb{R}^3}\left(|\nabla_\kappa u|^2+|\nabla_\nu \eta|^2\right) dx=0,
		$$
		which implies $ u=\eta\equiv0$. The proof is complete.
	\end{proof}

	\noindent {\bf Acknowledgments.}
	%The author would like to thank Professors Sining Zheng and Zhaoyin Xiang for some helpful communications.
	W. Wang was supported by National Key R\&D Program of China (No. 2023YFA1009200), NSFC under grant 12071054, and by Dalian High-level Talent Innovation Project (Grant 2020RD09).

\end{document}